\documentclass[12pt,times,twoside]{article}

\usepackage{graphicx,latexsym,euscript,makeidx,color}
\usepackage{amsmath,amsfonts,amssymb,amsthm,mathrsfs}
\usepackage{enumerate}
\usepackage[colorlinks,linkcolor=blue,anchorcolor=green,citecolor=red]{hyperref}%when print, use the next package
\usepackage{geometry}
       \def\oaB{\overleftarrow{B}}
\def\dbC{\mathbb{C}}     
     
\def\dbE{\mathbb{E}}     
\def\dbF{\mathbb{F}}   \def\cF{{\cal F}}  
     
\def\dbH{\mathbb{H}}

\def\dbP{\mathbb{P}}     
     
\def\dbR{\mathbb{R}}

\def\ss{\smallskip}                
\def\ms{\medskip}                
               
\def\ds{\displaystyle}           
        \def\q{\quad}                      
\def\ns{\noalign{\ss}}    \def\qq{\qquad}                    
    \def\hb{\hbox}                     
                   
         \def\rf{\eqref}                    
  \def\deq{\triangleq}               
            \def\({\Big (}
                  \def\){\Big )}
          \def\[{\Big[}
           \def\]{\Big]}
                   \def\cd{\cdot}

\def\a{\alpha}           \def\g{\gamma}      
\def\b{\beta}            \def\d{\delta}

\def\Om{\Omega}
\def\ba{\begin{array}}                \def\ea{\end{array}}
\def\bel{\begin{equation}\label}      \def\ee{\end{equation}}
\def\ben{\begin{enumerate}}           \def\een{\end{enumerate}}
\newtheorem{theorem}{Theorem}[section]
\newtheorem{definition}[theorem]{Definition}
\newtheorem{proposition}[theorem]{Proposition}
\newtheorem{corollary}[theorem]{Corollary}
\newtheorem{lemma}[theorem]{Lemma}
\newtheorem{remark}[theorem]{Remark}
\newtheorem{example}[theorem]{Example}

\makeatletter
   
   \@addtoreset{equation}{section}
\makeatother

\begin{document}
\title{\Large \bf Forward-backward doubly stochastic systems and classical solutions of path-dependent stochastic PDEs}

\author{
Yufeng Shi\thanks{Institute for Financial Studies and School of Mathematics, Shandong University,
Jinan 250100, China (Email: {\tt yfshi@sdu.edu.cn}).
This author is supported by National Key R\&D Program of China (Grant No. 2018YFA0703900), and National Natural Science Foundation of China (Grants No. 11871309 and 11371226).}~,~~~
Jiaqiang Wen\thanks{Department of Mathematics, Southern University of Science and Technology, Shenzhen, Guangdong, 518055, China (Email: {\tt wenjq@sustech.edu.cn}).
This author is supported by National Natural Science Foundation of China (Grant No. 12101291), Natural Science Foundation of Guangdong Province of China (Grant No. 2022A1515012017), and SUSTech start-up fund (Grant No. Y01286233).}~,~~~
Jie Xiong\thanks{Department of Mathematics and SUSTech International center for Mathematics, Southern University of Science and Technology, Shenzhen, Guangdong, 518055, China (Email: {\tt xiongj@sustech.edu.cn}).
This author is supported partially by National Natural Science Foundation of China (Grants No. 61873325 and 11831010), and SUSTech start-up fund (Grant No. Y01286120).}~,~~~
}
%
%\author
%{\textbf{Yufeng Shi}$^{\dagger}$,~~~~\textbf{Jiaqiang Wen}$^{\ddagger}$,~~~~\textbf{Jie Xiong}$^{\ddagger}$
%	\vspace{5mm} \\
%\normalsize{$^{\dagger}$Institute for Financial Studies and School of Mathematics,}\\
%\normalsize{Shandong University, Jinan 250100, China}\\
%\normalsize{$^{\ddagger}$Department of Mathematics and SUSTech International center for Mathematics,}\\ \normalsize{Southern University of Science and Technology, Shenzhen, 518055, China}\\
%}
%
\date{}
%
%\renewcommand{\thefootnote}{\fnsymbol{footnote}}
%
%\footnotetext[1]{Yufeng Shi is supported by the National Key R\&D Program of China (Grant No. 2018YFA0703900), the National Natural Science Foundation of China (Grant Nos. 11871309 and 11371226).
%Jiaqiang Wen is supported by China Postdoctoral Science Foundation (Grant No. 2019M660968) and
%Southern University of Science and Technology Start up fund Y01286233.
%Jie Xiong is supported by Southern University of Science and Technology Start up fund Y01286120 and the National Natural Science Foundation of China (Grants Nos. 61873325, 11831010).}
%
%\footnotetext[0]{E-mail addresses: yfshi@sdu.edu.cn (Y. Shi), wenjq@sustech.edu.cn (J. Wen),
%xiongj@sustech.edu.cn (J. Xiong).}

\maketitle

\begin{abstract}
In this paper, a class of non-Markovian forward-backward doubly stochastic systems is studied. By using the technique of functional It\^o (or path-dependent) calculus, the relationship between the systems and related path-dependent quasi-linear stochastic partial differential equations (SPDEs in short) is established, and the well-known nonlinear stochastic Feynman-Kac formula of Pardoux and Peng [Backward doubly stochastic differential equations and systems of quasilinear SPDEs, Probab. Theory Relat. Fields 98 (1994), pp. 209--227] is developed to the non-Markovian situation. Moreover, we obtain the differentiability of the solution to the forward-backward doubly stochastic systems and some properties of solutions to the path-dependent SPDEs.
\end{abstract}


\vspace{3mm}


\section{Introduction}

It is well-known that quasi-linear parabolic partial differential equations (PDEs in short) are related to Markovian forward-backward stochastic
differential equations behind Bismut \cite{Bismut} introduced the linear backward stochastic differential equations (BSDEs in short).
Then, after Pardoux and Peng \cite{Peng90} proved the existence and uniqueness for nonlinear BSDEs,
Peng \cite{Peng91} and Pardoux and Peng \cite{Peng92} introduced the nonlinear Feynman-Kac formula,
which provides a probabilistic representation for a large class of systems of quasi-linear PDEs
and generalizes the classical Feynman-Kac formula for linear parabolic PDEs.
Along this way, Pardoux and Peng \cite{Peng94} introduced a new class of BSDEs, called backward doubly stochastic differential equations (BDSDEs in short),
related to a class of quasi-linear parabolic backward stochastic PDEs (SPDEs in short).
They gave a probabilistic representation for solutions of such systems of semi-linear SPDEs,
and used it to prove the existence and uniqueness result of such SPDEs.
This result is summarized as the nonlinear stochastic Feynman-Kac formula,
which permits us to solve BDSDE by SPDE and, conversely, we can also use BDSDE to solve SPDE.
Since then, this important theory has attracted a lot of attention.
For example, Buckdahn and Ma \cite{Buckdahn01,Buckdahn012} introduced a definition of stochastic viscosity solution to the nonlinear SPDEs.
They also proved the existence and uniqueness of the stochastic viscosity solution,
and further extended the nonlinear stochastic Feynman-Kac formula.
The generalized BDSDEs and SPDEs with nonlinear Neumann boundary conditions have been investigated by Boufoussi, Casteren and  Mrhardy \cite{Boufoussi}.
Diehl and Friz \cite{Friz} established the connection between BSDEs with rough drivers and BDSDEs, and Xiong \cite{Xiong-13} obtained the strong uniqueness for the solution to a class of SPDEs with non-Lipschitz coefficients. Zhang and Zhao \cite{Zhang-Zhao07,Zhang-Zhao07} obtained the pathwise stationary solutions of SPDEs and BDSDEs on the infinite horizon under non-Lipschitz coefficients.
In application, Han, Peng and Wu \cite{Han-Peng-Wu-10} investigated the optimal control problems for backward doubly stochastic control systems, and obtained the maximum principle for backward doubly stochastic control systems. Shi, Wen and Xiong \cite{Shi-Wen-Xiong-20} developed BDSDEs of Volterra type, obtained the well-poseness of solutions, and studied the optimal control problems.
For more recent developments concerning BSDEs and BDSDEs, we refer the readers to \cite{Hu-Li-Wen-21,Wen-Li-Xiong 21,Zhang-Zhao10}, and the references therein.

\ms

During the past two decades, tremendous efforts have been made to extend the nonlinear stochastic Feynman-Kac formula to the non-Markovian situation. To the best of our knowledge, very few works have been done for nonlinear stochastic Feynman-Kac formula in non-Markovian case. The main difficulty is that, in the non-Markovian situation, one can not find a proper probabilistic representation for the solution of SPDE.

\ms

In this paper, we will overcome the difficulties to establish the nonlinear stochastic Feynman-Kac formula for non-Markovian BDSDEs of path-dependence, and obtain a proper probabilistic representation for the solution of related SPDEs.
The key technique of our approach is the functional It\^{o}/path-dependent calculus introduced by Dupire \cite{Dupire09} recently (see \cite{Cont10,Cont13,Zhang-17} and references cited therein). Dupire's main idea is to define new path derivatives, which provides an excellent tool for studying path-dependence.
This is an extension of the classical It\^{o} calculus to functionals depending on whole paths of a stochastic process instead of its current values only.
By using the functional It\^{o} calculus,
Ekren, Touzi and Zhang \cite{Ekren16-1,Ekren16-2}
introduced a new notion of viscosity solution and proved the comparison and the existence
for a class of fully nonlinear path-dependent PDEs by the dynamical nonlinear expectation theory.
Buckdahn, Ma and Zhang \cite{Buckdahn15} proposed a notion of pathwise viscosity solution for fully nonlinear forward SPDEs, and connected a forward SPDE with a path-dependent PDE.
%%Within the functional It\^{o} calculus framework, Jazaerli and Saporito \cite{Jazaerli-Saporito-17} introduced a measure of path-dependence of functionals.
Peng and Song \cite{Song15} established the existence and uniqueness theorem of Sobolev solutions to semi-linear (fully nonlinear) path-dependent PDEs.
Ren, Touzi and Zhang \cite{Ren-Touzi-Zhang-20}  provides a probabilistic proof of the comparison result for viscosity solutions of path-dependent semilinear PDEs.
In this non-Markovian situation, by using the technique of functional It\^{o} calculus,
the nonlinear Feynman-Kac formula was established by Ekren et al. \cite{Ekren14}, Peng and Wang \cite{Peng16}, and Ji and Yang \cite{Ji-Yang-13}.
Motivated by functional It\^{o} formulas,
Zong et al. \cite{Zong-Yin-Wang-Li-Zhang-18}
investigates asymptotic properties of systems represented by stochastic functional differential equations, and uses degenerate Lyapunov functionals to investigate the stability of general delay-dependent stochastic functional differential equations.

\ms

By virtue of the functional It\^{o} calculus, in this paper,
we will systematically study the nonlinear stochastic Feynman-Kac formula for non-Markovian BDSDEs, and research the relation between the non-Markovian BDSDEs and related path-dependent SPDEs.
In detail, consider the following system of non-Markovian forward-backward doubly stochastic differential equations,
\begin{align}
   X^{\gamma_{t}}(s)=&\ \gamma_{t}(t) + \int_t^s b(X_{r}^{\gamma_{t}}) dr + \int_t^s \sigma(X_{r}^{\gamma_{t}})dW(r), \q~ s\in[t,T], \label{SDE}\\
   Y^{\gamma_{t}}(s)=&\ \Phi(X_{T}^{\gamma_{t}}) + \int_s^T f(X_{r}^{\gamma_{t}},Y^{\gamma_{t}}(r),Z^{\gamma_{t}}(r)) dr \nonumber\\
   &\ + \int_s^T g(X^{\gamma_{t}}_{r},Y^{\gamma_{t}}(r),Z^{\gamma_{t}}(r)) d\oaB(r) - \int_s^T Z^{\gamma_{t}}(r) dW(r), \label{BDSDE}
\end{align}
where the integral with respect to $d\oaB$ is a backward It\^{o} integral, the integral with respect to $dW$ is a  forward It\^{o} integral, the coefficients $b(\cd),\sigma(\cd),\Phi(\cd),f(\cd)$ and $g(\cd)$ are given continuous functions, and $\gamma(\cd)$ is a given initial function with $\gamma_{t}$ denotes its path up to time $t$.

\ms

First, we study the differentiability of the solution to the BDSDE \rf{BDSDE}
by the method of ``frozen method" introduced by Peng \cite{Peng07}.
Second, by virtue of the classical nonlinear stochastic Feynman-Kac formula of Pardoux and Peng \cite{Peng94}, the path regularity of the process $Z$ is given.
Finally, we connect the non-Markovian BDSDE \rf{BDSDE} to the following path-dependent backward SPDE:
\begin{align}\label{0}
 u(\gamma_{t})=&\   \Phi(\gamma_{T}) + \int_t^T \big[ \mathcal{L}u(\gamma_{s})
                  + f(\gamma_{s},u(\gamma_{s}),(\sigma D_{x}u)(\gamma_{s})) \big] ds \nonumber\\
               &\ + \int_t^T g(\gamma_{s},u(\gamma_{s}),(\sigma D_{x}u)(\gamma_{s})) d\oaB(s), \q~ 0\leq t\leq T,
\end{align}
where
\begin{equation*}
  \mathcal{L}u = \frac{1}{2}tr\big[\big(\sigma\sigma^{T}\big)D_{xx}u\big] + \langle b,D_{x}u \rangle.
\end{equation*}
On the one hand, we prove that BDSDE \rf{BDSDE} can be seen as the solution of the quasi-linear path-dependent SPDE \rf{0}. Roughly speaking, the path-function $u(\gamma_{t})\deq Y^{\gamma_{t}}(t)$
is the unique smooth solution of the path-dependent SPDE \rf{0}, which provides the probabilistic representation for the solution of SPDEs in non-Markovian situation.
On the other hand,  we can also use the path-dependent SPDE \rf{0} to get the solution of BDSDE \rf{BDSDE}.
The results non-trivially develop the famous nonlinear stochastic Feynman-Kac formula of Pardoux and Peng \cite{Peng94} to non-Markovian situation, which provides a new way for the study of BDSDE theory and SPDE theory.
%
%Other results such as estimates and regularities for the solution of
%with respect to paths are also investigated. In particular, the probabilistic representation for the solution of path-dependent SPDE (\ref{0}) is obtained. The main technique here we will use is the functional It\^{o} calculus.
Recently, by using the technique of Malliavin calculus, Wen and Shi \cite{Wen} obtained a result for the nonlinear stochastic Feynman-Kac formula in the non-Markovian situation, where they consider the case of random coefficients.
Different with the tool of Malliavin calculus, we would like to emphasize that comparing with \cite{Wen}, the approaches and conclusions of this paper are of fundamental difference.

\ms

This paper is organized as follows.
In Section 2, some preliminary results about the functional It\^{o} calculus and BDSDEs are presented.
In Section 3, some regularities and estimates  results for the solutions of non-Markovian BDSDEs will be established. In Section 4, we prove the main results--the nonlinear stochastic Feynman-Kac formula in non-Markovian situation (Theorem \ref{2.2.1} and Theorem \ref{2.2.1-2}), which provide a one to one correspondence between non-Markovian BDSDEs and path-dependent PDEs.

\section{Preliminaries}

In this section, we recall some basic notions and results about functional It\^{o} calculus and BDSDEs,
which are needed in the sequels.
The reader may refer to the articles such as Dupire \cite{Dupire09}, Cont and Fourni\'{e} \cite{Cont13},
Ekren et al. \cite{Ekren14}, and Peng and Wang \cite{Peng16} for functional It\^{o} calculus and Pardoux and Peng \cite{Peng94} for BDSDEs.

\subsection{Functional It\^{o} calculus}

 Let $T>0$ be fixed.
 For each $t\in [0,T]$, we denote $\Lambda_{t}$ the set of bounded c\`{a}dl\`{a}g (right continuous with left limit) $\mathbb{R}^{d}$-valued functions on $[0,t]$ and
  $$\Lambda=\bigcup_{t\in[0,T]}\Lambda_{t}.$$
 For each $\gamma\in \Lambda$, its value at time $t$ is denoted by $\gamma(t)$ and its path up to time $t$ is
 denoted by $\gamma_{t}$, i.e., $\gamma_{t}=\gamma(r)_{0\leq r\leq t} \in \Lambda_{t}$.
 Then $\gamma(r)=\gamma_{t}(r)$, for  $r\in [0,t]$.
 For each $\gamma\in \Lambda, \ t\leq s$ and $x\in \mathbb{R}^{d}$, we denote
\begin{equation*}
  \begin{split}
    \gamma_{t}^{x}(r)\deq &\ \gamma(r)\mathbf{1}_{[0,t)}(r)+(\gamma(t)+x)\mathbf{1}_{\{t\}}, \q~ r\in [0,t],\\
      \gamma_{t,s}(r)\deq &\ \gamma(r)\mathbf{1}_{[0,t)}(r)+\gamma(t)\mathbf{1}_{[t,s]}(r), \q~ r\in[0,s].
  \end{split}
\end{equation*}
 It is clear that $\gamma_{t}^{x}\in \Lambda_{t}$ and $\gamma_{t,s}\in \Lambda_{s}$.
 For each $0\leq t\leq s\leq T$ and $\gamma,\overline{\gamma}\in \Lambda$, we denote
\begin{equation}\label{1.2.3}
\begin{split}
       \|\gamma_{t}\|\deq & \sup_{r\in[0,t]}|\gamma_{t}(r)|, \\
       \|\gamma_{t}-\overline{\gamma}_{s}\|\deq &\sup_{r\in[0,s]}|\gamma_{t,s}(r)-\overline{\gamma}_{s}(r)|, \\
       d_{\infty}(\gamma_{t},\overline{\gamma}_{s})\deq &\
\sup_{r\in[0,t\vee s]}|\gamma(r\wedge t) - \overline{\gamma}(r\wedge s)|+|t-s|^{\frac{1}{2}}.
\end{split}
\end{equation}
 It is obvious that $\Lambda_{t}$ is a Banach space with respect to $\|\cdot \|$.
%%However, since $\Lambda$ is not a linear space, $d_{\infty}$ is not a norm.
%%\ms
%%
%%\tb{NOT WELL-WRITTEN, DETAIL PLEASE! Now consider a functional $F$ on $(\Lambda,d_{\infty})$, i.e., $F:\Lambda\mapsto \mathbb{R}$.
%% The functional $F=F(\gamma_{t}^{x})$ can be regarded as a family of real valued functions:
%% \begin{equation*}
%%   F(\gamma_{t}^{x})= F(t,\gamma(r)_{0\leq r<t},\gamma(t)+x), \q~  \gamma_{t}\in \Lambda_{t}, \ t\in[0,T], \ x\in \mathbb{R}^{d}.
%% \end{equation*}
%% It is also important to understand $F(\gamma_{t}^{x})$ as a function of $t, \ \gamma(r)_{0\leq r<t}, \ \gamma(t)$ and $x$.
%% Now we shall introduce the path derivatives.}

\ms

 \begin{definition} \sl
 A functional $F$ defined on $\Lambda$ is said to be continuous at $\gamma_{t}\in \Lambda$,
 if for any $\varepsilon>0$ there exists $\delta>0$ such that
 for each $\overline{\gamma}_{\overline{t}}\in \Lambda$ with $d_{\infty}(\gamma_{t},\overline{\gamma}_{\overline{t}})<\delta$,
 we have $|F(\gamma_{t})-F(\overline{\gamma}_{\overline{t}})|<\varepsilon$.
 $F$ is said to be $\Lambda$-continuous if it is continuous at each $\gamma_{t}\in \Lambda$.
 \end{definition}

 \begin{definition} \sl
 Let $F:\Lambda\rightarrow \mathbb{R}$ be given. If there exists $p\in\mathbb{R}^{d}$, such that
\begin{align*}
F(\gamma_{t}^{x})=F(\gamma_{t}) + \langle p,x \rangle + o(|x|) \q~  as \q~  x\rightarrow 0, \q~  x\in \mathbb{R}^{d},
\end{align*}
then we say that $F$ is vertically differentiable at $\gamma_{t}\in \Lambda$ and define $D_{x}F(\gamma_{t})\deq p$.
A functional $F$ is said to be vertically differentiable if $D_{x}F(\gamma_{t})$ exists for each $\gamma_{t}\in \Lambda$.
We can similarly define the Hessian $D_{xx}F(\gamma_{t})$.
It is an $\mathbb{S}(d)$-valued function defined on $\Lambda$, where $\mathbb{S}(d)$ is the space of all $d \times d$ symmetric matrices.
\end{definition}

\begin{definition} \sl
 Let $F:\Lambda\rightarrow \mathbb{R}$ and $\gamma_{t}\in \Lambda$ be given. If there is $a\in\mathbb{R}$, such that
$$
 F(\gamma_{t,s})=F(\gamma_{t}) + a(s-t) + o(|s-t|) \q~ as \q~  s\rightarrow t, \q~  s\geq t,
$$
 then we say $F$ is horizontally differentiable at $\gamma_{t}$ and define $D_{t}F(\gamma_{t})\deq a$.
A functional $F$ is said to be horizontally differentiable if $D_{t}F(\gamma_{t})$ exists for each $\gamma_{t}\in \Lambda$.
\end{definition}

The following space shall be used frequently in the following sections.

 \begin{definition} \sl
 $F$ is said to be in $\mathbb{C}^{1,2}_{l,\text{Lip}}(\Lambda)$, if $D_{t}F$, $D_{x}F$, and $D_{xx}F$ exist and we have
 \begin{equation*}
        |\varphi(\gamma_{t}) - \varphi(\overline{\gamma}_{\overline{t}})|
   \leq C(1+\|\gamma_{t}\|^{m}+\|\overline{\gamma}_{\overline{t}}\|^{m}) d_{\infty}(\gamma_{t},\overline{\gamma}_{\overline{t}}),
        \q~  \forall  \gamma_{t}, \overline{\gamma}_{\overline{t}} \in \Lambda,
 \end{equation*}
 where $\varphi=F$, $D_{t}F$, $D_{x}F$, $D_{xx}F$ and $C$ and $m$ are two constants depending only on $\varphi$.
 Similarly, we can define $\mathbb{C}^{0,1}_{l,\text{Lip}}(\Lambda)$, $\mathbb{C}^{0,2}_{l,\text{Lip}}(\Lambda)$.
 \end{definition}

 \begin{example} \rm
  If $F(\gamma_{t})=f(t,\gamma(t))$ with $f\in \mathbb{C}^{1,2}([0,T]\times \mathbb{R})$, then the path-derivatives reduce to the classical derivatives:
  \begin{equation*}
     D_{t}F= \partial _{t}f, \q~  D_{x}F= \partial _{x}f, \q~  D_{xx}F= \partial _{xx}f.
  \end{equation*}
  In general, these derivatives also satisfy the classical properties: Linearity, product and chain rules.
 \end{example}
The following functional It\^{o}'s formula was firstly obtained by Dupire \cite{Dupire09} and then generalized by Cont and Fourni\'{e} \cite{Cont10,Cont13}.
\begin{theorem}\label{1} \sl
Let $(\Om,\cF,(\cF_t)_{t\in[0,T]},\dbP)$ be a probability space, if $X$ is a continuous semi-martingale
and $F$ is in $\mathbb{C}^{1,2}_{l,\text{Lip}}(\Lambda)$, then for any $t\in [0,T)$,
\begin{align*}
  F(X_{t})=\ F(X_{0})+ \int_0^t D_{s}F(X_{s})ds
            +\int_0^t D_{x}F(X_{s})dX(s)
            +\frac{1}{2}\int_0^t {\rm tr}[D_{xx}F(X_{s})]d\langle X\rangle(s).
\end{align*}
\end{theorem}

%\begin{remark} \rm
%The approach to the above Lemma \ref{1} is different from the one of Cont and Fourni\'{e} \cite{Cont10} (see Theorem 3 of \cite{Cont10}). Indeed, similar to Theorem 4.1 of Cont and Fourni\'{e} \cite{Cont13}, the approach to the above Lemma \ref{1} is probabilistic, and the proof to Theorem 3 of Cont and Fourni\'{e} \cite{Cont10} is non-probabilistic.
%\end{remark}

\subsection{BDSDEs}

%%\tb{Note-Uniform. 1 Space, 2 Norm}\\

%%%%%%%%%%%%%
Let $(\Omega,\cF,\dbF,\dbP)$ be a complete filtered probability space and $T>0$ be a fixed terminal time.
Let $\{W(t);0\leq t\leq T\}$ and $\{B(t);0\leq t\leq T\}$ be two mutually independent standard
Brownian motions defined on $(\Omega,\mathcal{B}(\Omega),\dbP)$ with values in $\mathbb{R}^{d}$ and in $\mathbb{R}^{l}$, respectively.
Let $\mathcal{N}$ denote the class of $P$-null sets in $\Omega$.
For each $t\in [0,T]$, we define
 \begin{equation*}
 \mathcal{F}_{t} \deq   \mathcal{F}_{t}^{W} \vee \mathcal{F}_{t,T}^{B},
 \end{equation*}
 where for any process $\{\eta(t)\}, \ \mathcal{F}_{s,t}^{\eta}=\sigma\{\eta(r)-\eta(s);s\leq r\leq t\}\vee \mathcal{N}$ and
 $\mathcal{F}_{t}^{\eta}=\mathcal{F}_{0,t}^{\eta}$.
Note that the collection $\{ \mathcal{F}_{t}; t\in[0,T] \}$ is neither increasing nor decreasing, and it does not constitute a filtration.

\ms

For any $t\in[0,T]$, we denote by $M^{2}(0,T; \dbR^n)$ the set of  $n$-dimensional
$\mathcal{F}_{t}$-measurable stochastic process
$\{ \varphi(t); t\in[0,T] \}$ which satisfies $\dbE\int_0^T |\varphi(t)|^{2} dt<\infty$.
%
%\begin{enumerate}[~~\,\rm(i)]
%%
%\item [(i)] $\|\varphi \|_{M^{2}}^{2} \deq E\int_0^T |\varphi(t)|^{2} dt<\infty$;
%%
%\item [(ii)] $\varphi(t)$ is $\mathcal{F}_{t}$-measurable, for any $t\in[0,T].$
%%
%\end{enumerate}
%%
Similarly, we denote by $S^{2}(0,T; \dbR^n)$ the set of $n$-dimensional continuous stochastic process $\{ \varphi(t); t\in[0,T] \}$ which is $\mathcal{F}_{t}$-measurable and satisfies $\dbE(\sup\limits_{0\leq t\leq T} |\varphi(t)|^{2})<\infty$.
%
%\begin{enumerate}[~~\,\rm(i)]
%%
%\item [(i)] $\|\varphi \|_{S^{2}}^{2} \deq E(\sup\limits_{0\leq t\leq T} |\varphi(t)|^{2})<\infty$;
%%
%\item [(ii)] $\varphi(t)$ is $\mathcal{F}_{t}$-measurable, for any $t\in[0,T]$.
%%
%\end{enumerate}
%

\ms

Now consider the following assumptions:
\begin{itemize}
  \item [(H1)] Let $b:\Lambda \rightarrow \mathbb{R}^{d}$ and $\sigma:\Lambda \rightarrow \mathbb{R}^{d\times d}$ be two given continuous functions,
        and there exists a constant $L\geq 0$ and $m\geq 0$, such that for each $\gamma_{t},\overline{\gamma}_{t}\in\Lambda$,
  \begin{equation*}
        |b(\gamma_{t}) - b(\overline{\gamma}_{t})| + |\sigma(\gamma_{t}) - \sigma(\overline{\gamma}_{t})|
   \leq L (1+\|\gamma_{t}\|^{m}+\|\overline{\gamma}_{t}\|^{m})\|\gamma_{t}-\overline{\gamma}_{t}\|.
  \end{equation*}
  \item [(H2)] Suppose $ f:\Lambda\times  \mathbb{R}^{k}\times \mathbb{R}^{k\times d}\longrightarrow  \mathbb{R}^{k}$ and
        $g:\Lambda\times  \mathbb{R}^{k}\times \mathbb{R}^{k\times d}\longrightarrow  \mathbb{R}^{k\times l}$ are two given continuous functions.
        Moreover, there exist constants $m\geq 0$, $C > 0$ and $0 <\alpha < 1$ such that for every
        $\gamma_{t},\overline{\gamma}_{t}\in\Lambda, \ y,\overline{y}\in \mathbb{R}^{k}, z,\overline{z}\in \mathbb{R}^{k\times d}$,
$$\left\{\ba{ll}
\ds |f(\gamma_{t},y,z)-f(\overline{\gamma}_{t},\overline{y},\overline{z})|\leq C\big( (1+\|\gamma_{t}\|^{m}+\|\overline{\gamma}_{t}\|^{m})\|\gamma_{t}-\overline{\gamma}_{t}\|+|y-\overline{y}| + |z-\overline{z}|\big);\\
\ns\ns\ds |g(\gamma_{t},y,z)-g(\overline{\gamma}_{t},\overline{y},\overline{z})|\leq C\big( (1+\|\gamma_{t}\|^{m}+\|\overline{\gamma}_{t}\|^{m})\|\gamma_{t}-\overline{\gamma}_{t}\|+|y-\overline{y}|\big) + \alpha |z-\overline{z}|.
\ea\right.$$
\end{itemize}

The existence and uniqueness of solutions for classical SDEs and BDSDEs are by now well known, for their proofs the reader is referred to Lipster and Shiryaev \cite{Lipster78}, and Pardoux and Peng \cite{Peng94}. The following Proposition \ref{0.1.1} and Proposition \ref{0.1.2} are the existence and uniqueness of solutions to SDEs and BDSDEs of path-dependent version, whose proofs are the same as in \cite{Lipster78} and \cite{Peng94} without essential difference. Moreover, Proposition \ref{1.8.8} is Lemma 1.3 of \cite{Peng94}, and Proposition \ref{1.8.9} is the path-dependent version of Theorem 3.1 of Shi et al. \cite{Shi-Gu-Liu-05}.

\begin{proposition}\label{0.1.1} \sl
Under the condition (H1), for each $\gamma_{t}\in \Lambda, \ t\in[0,T]$, the following SDE admits a unique solution in $M^{2}(0,T; \mathbb{R}^{d})$,
\begin{equation*}
   X^{\gamma_{t}}(s)=\gamma_{t}(t) + \int_t^s b(X_{r}^{\gamma_{t}}) dr + \int_t^s \sigma(X_{r}^{\gamma_{t}})dW(r), \q~ s\in[t,T].
\end{equation*}
\end{proposition}

\begin{proposition}\label{0.1.2} \sl
Under the condition (H2), for each $\xi\in$ $L^{2}(\Omega,\mathcal{F}_{T},\dbP;\dbR^k)$ and for every $t\in[0,T]$, the following BDSDE
\begin{equation*}
   Y^{\gamma_{t}}(t)=\xi + \int_t^T f(X_{s}^{\gamma_{t}},Y^{\gamma_{t}}(s),Z^{\gamma_{t}}(s)) ds
   + \int_t^T g(X^{\gamma_{t}}_{s},Y^{\gamma_{t}}(s),Z^{\gamma_{t}}(s)) d\oaB(s) - \int_t^T Z^{\gamma_{t}}(s) dW(s),
\end{equation*}
 has a pair of unique solution $(Y,Z)$ in $S^{2}(0,T; \mathbb{R}^{k})\times M^{2}(0,T; \mathbb{R}^{k\times d})$.
\end{proposition}

\begin{proposition}\label{1.8.8} \sl
Let $\Phi\in C^2(\dbR^k)$, $\alpha \in S^{2}(0,T;\mathbb{R}^{k})$, $\beta\in M^{2}(0,T;\mathbb{R}^{k})$,
$\gamma\in M^{2}(0,T;\mathbb{R}^{k\times l})$, $\delta\in M^{2}(0,T;\mathbb{R}^{k\times d})$ be such that
\begin{align*}
  \a(t)=\a(0) + \int_0^t \b (s) ds + \int_0^t \g(s) d\oaB(s) + \int_0^t \d(s) dW(s), \q~ t\in[0,T].
\end{align*}
Then
\begin{align*}
  \Phi(\a(t))=&\ \Phi(\a(0))+\int_0^t \big\langle \Phi'(\a(s)),\b (s)\big\rangle ds
             +\int_0^t \big\langle \Phi'(\a(s)),\g(s) d\oaB(s)\big\rangle
             + \int_0^t \big\langle \Phi'(\a(s)),\d(s)dW(s)\big\rangle\\
           & -\frac{1}{2}\int_0^t tr\big[\Phi''(\a(s))\g(s)\g(s)^\top\big] ds
             +\frac{1}{2}\int_0^t tr\big[\Phi''(\a(s)))\d(s)\d(s)^\top\big] ds.
\end{align*}
\end{proposition}
%
%We also recall the following comparison theorem of BDSDE (see \cite{Shi-Gu-Liu-05}).
%
\begin{proposition}\label{1.8.9} \sl
Let $k=1$. Given the generators $f_1$, $f_2$ and $g$ satisfying (H2), and two terminal value $\xi_1,\xi_2\in L^{2}(\Omega,\mathcal{F}_{T},\dbP;\dbR)$. We denote by $(Y_1,Z_1)$ and $(Y_2,Z_2)$ the solution of BDSDE with the generator $(\xi_1,f_1,g)$ and $(\xi_2,f_2,g)$, respectively. If $\xi_1\geq\xi_2$, a.s., and  $f_1\geq f_2$,  then  $Y_1(t)\geq Y_2(t)$, a.s. for each $t\in[0,T]$.
\end{proposition}

\section{Non-Markovian BDSDEs}

In this section, we study the properties of solutions to non-Markovian BDSDEs by a stochastic calculus approach, which is helpful for us to study the path-dependent SPDEs. First, let us recall some notations which will be used later.
\begin{itemize}
  \item [$\bullet$] $\dbC_{b,\text{Lip}}^{n}(\mathbb{R}^{p};\mathbb{R}^{q})$: The space of all bounded functions of class $C^{n}(\mathbb{R}^{p};\mathbb{R}^{q})$
        such that all their partial derivatives up to order $n$ are bounded Lipschtiz continuous functions.
  \item [$\bullet$] $\dbC_{l,\text{Lip}}(\mathbb{R}^{p};\mathbb{R}^{q})$:
        The space of all $\mathbb{R}^{q}$-valued and continuous functions $\varphi$ defined on
        $\mathbb{R}^{p}$ such that
        \begin{equation*}
        |\varphi(x) - \varphi(y)| \leq C(1+|x|^{m}+|y|^{m})|x-y|, \q~ \forall x,y\in \mathbb{R}^{p},
        \end{equation*}
        where $C,m$ are two constants depending only on $\varphi$.
  \item [$\bullet$] $\dbC_{l,\text{Lip}}^{n}(\mathbb{R}^{p};\mathbb{R}^{q})$: The space of all functions of class $C^{n}(\mathbb{R}^{p};\mathbb{R}^{q})$
        such that all their partial derivatives up to order $n$ are in $\dbC_{l,\text{Lip}}(\mathbb{R}^{p};\mathbb{R}^{q})$.
  \item [$\bullet$] $\dbC_{l,\text{Lip}}^{n,\bar{n}}(\mathbb{R}^{p}\times \mathbb{R}^{\bar{p}};\mathbb{R}^{q})$:
        The space of all $\mathbb{R}^{q}$-valued continuous functions $\varphi(x,y)$ defined on $\mathbb{R}^{p}\times \mathbb{R}^{\bar{p}}$ such that
        $x\mapsto \varphi(x,y)\in C^{n}(\mathbb{R}^{p};\mathbb{R}^{q})$,  $y\mapsto \varphi(x,y)\in C^{\bar{n}}(\mathbb{R}^{\bar{p}};\mathbb{R}^{q})$
        and all these partial derivatives belong to $\dbC_{l,\text{Lip}}(\mathbb{R}^{p}\times \mathbb{R}^{\bar{p}};\mathbb{R}^{q})$.
\end{itemize}

%The following directional derivatives will be used frequently in the sequel.

\begin{definition} \sl
Given $\mathbb{R}^{k}$-valued function $\Phi$ defined on $\Lambda_{T}$.
We say that $\Phi$ is in $\dbC^{1}(\Lambda_{T})$ if for each $\gamma\in \Lambda_{T}$ and $t\in [0,T]$,
there exists a vector $p=(p_{1},...,p_{k})\in \mathbb{R}^{k\times d}$ such that for each $i\leq k$,
\begin{equation*}
  \Phi_{i}(\gamma_{\gamma_{t}^{x}})-\Phi_{i}(\gamma) = \langle p_{i},x \rangle + o(|x|) \q as \q x\rightarrow 0, \ \ x\in \mathbb{R}^{d},
\end{equation*}
where $\gamma_{\gamma_{t}^{x}}(r)= \gamma(r)\mathbf{1}_{[0,t)}(r) + (\gamma(r)+x)\mathbf{1}_{[t,T]}(r)$.
Then we denote $\Phi_{t}'(\gamma)\deq p$. Similarly, the space $\dbC^{2}(\Lambda_{\bar{T}})$ could be defined for each $\bar T\in[0,T].$
\end{definition}

\begin{definition} \rm
We say $\Phi\in \dbC^{2}(\Lambda_{T})$ is in $\dbC_{l,\text{Lip}}^{2}(\Lambda_{T})$
if there exist two constants $C\geq 0$ and $m\geq 0$ depending only on $\Phi$ such that for each $\gamma, \overline{\gamma}\in \Lambda_{T}, r,s,\in [0,T],$
\begin{equation*}
  \begin{cases}
   |\phi(\gamma)-\phi(\overline{\gamma})| \leq C(1+\|\gamma\|^{m}+\|\overline{\gamma}\|^{m}) \|\gamma-\overline{\gamma}\|, \\
|\phi_{r}'(\gamma)-\phi_{s}'(\overline{\gamma})|+|\phi_{r}''(\gamma)-\phi_{s}''(\overline{\gamma})|
   \leq C(1+\|\gamma\|^{m}+\|\overline{\gamma}\|^{m})\big( |r-s| + \|\gamma-\overline{\gamma}\| \big).
  \end{cases}
\end{equation*}
The spaces $\dbC_{l,\text{Lip}}(\Lambda_{\bar T})$ and  $\dbC_{l,\text{Lip}}^{1}(\Lambda_{\bar T})$ can be defined similarly for each $\bar T\in[0,T].$
\end{definition}

In the following, we shall make the following assumptions.

\begin{itemize}
  \item [(H3)] $\Phi$ is a $\mathbb{R}^{k}$-valued function defined on $\Lambda_{T}$.
        Moreover $\Phi\in \dbC_{l,\text{Lip}}^{2}(\Lambda_{T})$ with the Lipschitz coefficients  $C$ and $m$.
  \item [(H4)] $f(\gamma_{t},y,z):\Lambda\times  \mathbb{R}^{k}\times \mathbb{R}^{k\times d}\longrightarrow  \mathbb{R}^{k}$ and
        $g(\gamma_{t},y,z):\Lambda\times  \mathbb{R}^{k}\times \mathbb{R}^{k\times d}\longrightarrow  \mathbb{R}^{k\times l}$ are two given continuous functions. For any $\gamma_{t}\in \Lambda$ and
        $s\in[0,t]$, $(x,y,z)\mapsto \big(f((\gamma_{t})_{\gamma_{s}^{x}},y,z),$ $ g((\gamma_{t})_{\gamma_{s}^{x}},y,z) \big)$ is of class
        $\dbC_{l,\text{Lip}}^{2}(\mathbb{R}^{d}\times\mathbb{R}^{k}\times\mathbb{R}^{k\times d};\mathbb{R}^{k})$ and         $\dbC_{l,\text{Lip}}^{2}(\mathbb{R}^{d}\times\mathbb{R}^{k}\times\mathbb{R}^{k\times d};\mathbb{R}^{k\times l})$ respectively,
        and all of their derivatives are bounded and continuous in $t$;
        for and $(y,z)$, $\gamma_t\mapsto \big(f(\gamma_t,y,z),g(\gamma_t,y,z)\big)$ is of class
        $\dbC_{l,\text{Lip}}^{2}(\Lambda_{t})$ and all of their derivatives are bounded and continuous in $t$.
        Moreover, there exists $C\geq 0$ and $0<\a<1$ such that $gg^{\top}(\gamma_{t},y,z)\leq \a zz^\top + C(|g(\gamma_{t},0,0)|^{2}+|y|^{2})I$ and
        $g_{z}'(\gamma_{t},y,z)\theta\theta^\top g_{z}'(\gamma_{t},y,z)^\top \leq \theta\theta^\top $,
        for all $\gamma_{t}\in\Lambda, y\in \mathbb{R}^{k}, z,\theta\in\mathbb{R}^{k\times d}$.
  \item [(H5)] The function $f(\gamma_{t},y,z)=  \overline{f}(t,\gamma(t),y,z)$ and
        $g(\gamma_{t},y,z)=  \overline{g}(t,\gamma(t),y,z)$,
        where $\overline{f}\in \dbC_{l,\text{Lip}}^{0,2}([0,T]$ $\times\mathbb{R}^{d}\times\mathbb{R}^{k}\times\mathbb{R}^{k\times d};\mathbb{R}^{k})$
        and $\overline{g}\in \dbC_{l,\text{Lip}}^{0,2}([0,T]\times\mathbb{R}^{d}\times\mathbb{R}^{k}\times\mathbb{R}^{k\times d};\mathbb{R}^{k\times l})$ and all of their derivatives are bounded and continuous in $t$.
        Moreover, there exists two constants $C\geq 0$ and $0<\a<1$ such that $gg^{\top}(\gamma_{t},y,z)\leq \a zz^\top + C(|g(\gamma_{t},0,0)|^{2}+|y|^{2})I$ and
        $g_{z}'(\gamma_{t},y,z)\theta\theta^\top g_{z}'(\gamma_{t},y,z)^\top$ $\leq \theta\theta^\top $,
        for all $\gamma_{t}\in\Lambda, y\in \mathbb{R}^{k}, z,\theta\in\mathbb{R}^{k\times d}$.
\end{itemize}

%%\tb{Why there is (H4) and (H5). They look very similar. Explain this point.}

\begin{remark} \rm
The assumptions (H3-H5) are path-dependent versions of those in Pardoux and Peng \cite{Peng94},
while (H5) is a bit stronger than that in \cite{Peng94}.
In order to establish the existence of Dupire's path derivative, we point out that the functions $\Phi$, $f$ and $g$ are needed to be defined on the space of c\`{a}dl\`{a}g paths.
\end{remark}

Note that under assumptions (H1-H4), which we will assume throughout this article, SDE (\ref{SDE}) and BDSDE (\ref{BDSDE}) have unique solution.

\subsection{Regularity of solution of BDSDE}

In this subsection, we make some the estimates of the solution of BDSDE (\ref{BDSDE}).
Assume the Lipschitz constants with respect to $b,\sigma,\Phi,f,g$ are $C$, $m$ and $\a$.
The following result concerning the estimates of solution of BDSDE (\ref{BDSDE}), is useful for our later discussion.

\begin{lemma}\label{26} \sl
For any $p\geq 2$, there exist two constants $C_{p}$ and $q$ depending only on $C,T,m, \a$ and $p$ such that
\begin{align}
  &E\bigg[\sup_{s\in[t,T]}|Y^{\gamma_{t}}(s)|^{p}\bigg]\leq  C_{p} \big(1+\| \gamma_{t}\|^{q}\big), \label{21}\\
  &E\bigg[\bigg| \int_t^T |Z^{\gamma_{t}}(s)|^{2} ds \bigg|^{\frac{p}{2}}\bigg]\leq  C_{p} \big(1+\| \gamma_{t}\|^{q}\big) \label{22}.
\end{align}
\end{lemma}

\begin{proof}
The proof is similar to the proof of Theorem 1.4 in Pardoux and Peng \cite{Peng94},  and to simplify the presentation,  we only consider the case of $d=k=1$ and $p=2$. First, using the classical method of SDEs theory, there is a constant $C_{0}$ (which may be change line by line) such that
\begin{equation}\label{19}
  E\bigg[\sup_{s\in[t,T]}|X^{\gamma_{t}}(s)|^{2} \bigg] \leq C_{0} \big(1+\| \gamma_{t}\|^{2m}\big).
\end{equation}
Indeed, applying It\^{o}'s formula to $|X^{\gamma_{t}}(r)|^{2}$ on $[t,s]$ yields that
\begin{align*}
   |X^{\gamma_{t}}(s)|^{2}=&|\gamma_{t}(t)|^{2}+
   2\int_t^s b(X_{r}^{\gamma_{t}})X^{\gamma_{t}}(r) dr
   +2\int_t^s \sigma(X_{r}^{\gamma_{t}}) X^{\gamma_{t}}(r)dW(r)+ \int_t^s \sigma^{2}(X_{r}^{\gamma_{t}}) dr.
\end{align*}
By inequality $2ab\leq a^{2}+b^{2}$ and Burkholder-Davis-Gundy's inequality, we obtain
\begin{align*}
   E\bigg[\sup_{r\in[t, s]}|X^{\gamma_{t}}(r)|^{2}\bigg]
   \leq&C_{0}\bigg[ |\gamma_{t}(t)|^{2} + E\int_t^s \big(|b(X_{r}^{\gamma_{t}})|^{2}+|\sigma(X_{r}^{\gamma_{t}})|^{2}\big) dr
        +  E\int_t^s |X^{\gamma_{t}}(r)|^{2} dr\bigg].
\end{align*}
Then, from assumption (H1) and Gronwall's inequality, we obtain (\ref{19}).

\ms

Second, we prove (\ref{21}) and (\ref{22}). Using It\^{o}'s formula (Proposition \ref{1.8.8}) to $|Y^{\gamma_{t}}(r)|^{2}$ on $[s,T]$, one has
\begin{align}
  &|Y^{\gamma_{t}}(s)|^{2} + \int_s^T |Z^{\gamma_{t}}(r)|^{2} dr \nonumber\\
&=|\Phi(X_{T}^{\gamma_{t}})|^{2} + 2\int_s^T Y^{\gamma_{t}}(r)f(X^{\gamma_{t}}_{r},Y^{\gamma_{t}}(r),Z^{\gamma_{t}}(r)) dr
  + 2\int_s^T Y^{\gamma_{t}}(r)g(X^{\gamma_{t}}_{r},Y^{\gamma_{t}}(r),Z^{\gamma_{t}}(r)) d\oaB(r)  \nonumber\\
 &\q - 2\int_s^T Y^{\gamma_{t}}(r)Z^{\gamma_{t}}(r) dW(r)
  + \int_s^T |g(X^{\gamma_{t}}_{r},Y^{\gamma_{t}}(r),Z^{\gamma_{t}}(r))|^{2} dr.\label{23}
\end{align}
Hence
\begin{align*}
  &E|Y^{\gamma_{t}}(s)|^{2} + E\int_s^T |Z^{\gamma_{t}}(r)|^{2} dr \nonumber\\
&=E|\Phi(X_{T}^{\gamma_{t}})|^{2} + 2E\int_s^T Y^{\gamma_{t}}(r)f(X^{\gamma_{t}}_{r},Y^{\gamma_{t}}(r),Z^{\gamma_{t}}(r)) dr
  + E\int_s^T |g(X^{\gamma_{t}}_{r},Y^{\gamma_{t}}(r),Z^{\gamma_{t}}(r))|^{2} dr.
\end{align*}
%Note that we can conclude from (H4) that for any $\alpha\leq \alpha'\leq 1$, there exists $C(\alpha')$ such that
%\begin{align}
%|g(\gamma_{t},y,z)|^{2} \leq C(\alpha')(|g(\gamma_{t},0,0)|^{2} + |y|^{2}) + \alpha'|z|^{2}. \label{24}
%\end{align}
Then from (H4), and using H\"{o}lder's and Young's inequalities, we obtain that there exists a constant $\theta>0$ such that for any $s\in [t,T]$,
\begin{align*}
  &E|Y^{\gamma_{t}}(s)|^{2} + \theta E\int_s^T |Z^{\gamma_{t}}(r)|^{2} dr \nonumber\\
&\leq E|\Phi(X_{T}^{\gamma_{t}})|^{2}  + C_{0} E\int_s^T \big[|f(X^{\gamma_{t}}_{r},0,0)|^{2}+|g(X^{\gamma_{t}}_{r},0,0)|^{2}\big] dr
  + C_{0} E\int_s^T |Y^{\gamma_{t}}(r)|^{2} dr.
\end{align*}
It then follows from (H3), (H4) and (\ref{19}), and by using Gronwall's inequality that
\begin{align*}
  \sup_{s\in[t,T]} E|Y^{\gamma_{t}}(s)|^{2} + E\int_t^T |Z^{\gamma_{t}}(r)|^{2} dr
\leq C_{p} \big(1+\| \gamma_{t}\|^{2m}\big).
\end{align*}
Therefore, we obtain (\ref{22}).
On the other hand, by applying the same inequalities to Eq. (\ref{23}),
\begin{align*}
  |Y^{\gamma_{t}}(s)|^{2}
\leq &|\Phi(X_{T}^{\gamma_{t}})|^{2}
      + C_{0}\int_s^T  \big[|f(X^{\gamma_{t}}_{r},0,0)|^{2}+|g(X^{\gamma_{t}}_{r},0,0)|^{2} + |Y^{\gamma_{t}}(r)|^{2} \big] dr \nonumber\\
     &+ 2\int_s^T Y^{\gamma_{t}}(r)g(X^{\gamma_{t}}_{r},Y^{\gamma_{t}}(r),Z^{\gamma_{t}}(r)) d\oaB(r)
      - 2\int_s^T Y^{\gamma_{t}}(r)Z^{\gamma_{t}}(r) dW(r).
\end{align*}
Hence, from Burkholder-Davis-Gundy's inequality,
\begin{align}
       E\bigg[\sup_{s\in[t,T]}|Y^{\gamma_{t}}(s)|^{2}\bigg]
 \leq& E|\Phi(X_{T}^{\gamma_{t}})|^{2}
      + C_{0}E\int_t^T  \big[|f(X^{\gamma_{t}}_{r},0,0)|^{2}+|g(X^{\gamma_{t}}_{r},0,0)|^{2} + |Y^{\gamma_{t}}(r)|^{2} \big] dr
        \nonumber\\
     &+C_{0}E \sqrt{\int_t^T |Y^{\gamma_{t}}(r)|^{2}|g(X^{\gamma_{t}}_{r},Y^{\gamma_{t}}(r),Z^{\gamma_{t}}(r))|^{2} dr} \label{24}\\
     &+C_{0}E\sqrt{\int_t^T |Y^{\gamma_{t}}(r)|^{2}|Z^{\gamma_{t}}(r)|^{2} dr}. \label{25}
\end{align}
We estimate the term (\ref{25}) as follows:
\begin{align*}
 E\sqrt{\int_t^T |Y^{\gamma_{t}}(r)|^{2}|Z^{\gamma_{t}}(r)|^{2} dr}
 \leq& E \bigg[\sup_{s\in[t,T]}|Y^{\gamma_{t}}(s)| \sqrt{\int_t^T |Z^{\gamma_{t}}(r)|^{2} dr} \bigg]\\
 \leq& \frac{\delta}{2} E \bigg[\sup_{s\in[t,T]}|Y^{\gamma_{t}}(s)|^{2}\bigg]
     + \frac{1}{2\delta}  E \bigg[\int_t^T |Z^{\gamma_{t}}(r)|^{2} dr \bigg],
\end{align*}
where $\delta>0$ is a constant, and we choose it small enough so that $C\delta<\frac{1}{3}$.
The term (\ref{24}) can be treated analogously, and we deduce that
\begin{align*}
E\bigg[\sup_{s\in[t,T]}|Y^{\gamma_{t}}(s)|^{2}\bigg]\leq  C_{p} \big(1+\| \gamma_{t}\|^{2m}\big).
\end{align*}
Similarly, we can get the same result for $p\geq 2.$ Finally, for multi-dimensional case, one can similarly obtain the results without substantial difficulty. This completes the proof.
\end{proof}

Based on the above lemma, one could obtain the following general result for the estimates of the solution of BDSDEs.

\begin{lemma}\label{26.2} \sl
Suppose that  functions $u$ and $v$ are in $M^{2}(0,T; \dbH)$. Let $(Y^{\gamma_{t}}(\cd),Z^{\gamma_{t}}(\cd))$
be the solution of the following BDSDE:
\begin{align*}
   Y^{\gamma_{t}}(s)=&\ \Phi(X_{T}^{\gamma_{t}}) + \int_s^T \big[f(X_{r}^{\gamma_{t}},Y^{\gamma_{t}}(r),Z^{\gamma_{t}}(r))+u(r)\big] dr \\
   &\ + \int_s^T \big[g(X_{r}^{\gamma_{t}},Y^{\gamma_{t}}(r),Z^{\gamma_{t}}(r))+v(r)\big] d\oaB(r) - \int_s^T Z^{\gamma_{t}}(r) dW(r),
\end{align*}
where $(X^{\gamma_{t}}(\cd))$ is the adapted solution of SDE \rf{SDE}. Then for every $p\geq 2$, there exist two constants $C_{p}$ and $q$ depending only on $C,T,m, \a$ and $p$ such that
\begin{align}
  &E\bigg[\sup_{s\in[t,T]}|Y^{\gamma_{t}}(s)|^{p}+\bigg| \int_t^T |Z^{\gamma_{t}}(s)|^{2} ds \bigg|^{\frac{p}{2}}\bigg] \nonumber\\
  &\leq  C_{p} \bigg(1+\| \gamma_{t}\|^{q}+E\bigg[\bigg|\int_t^T |u(s)| ds \bigg|^{p}\bigg]
  +E\bigg[\bigg|\int_t^T |v(s)|^{2} ds \bigg|^{\frac{p}{2}}\bigg]\bigg). \label{21.2}
\end{align}
\end{lemma}

\begin{proof}
For $s\in[t,T]$, for simplicity of presentation, we denote
$$U(s)=\int_t^su(r)dr,\qq V(s)=\int_t^sv(r)d\oaB(r).$$
Then, it is easy to see that $(Y^{\gamma_{t}}(s)+U(s)+V(s),Z^{\gamma_{t}}(s))$ is the solution of the following BDSDE:
\begin{align*}
   \bar Y^{\gamma_{t}}(s)=&\ \Phi(X_{T}^{\gamma_{t}}) + U(T)+V(T)+ \int_s^T
   \bar f(X_{r}^{\gamma_{t}},\bar Y^{\gamma_{t}}(r),\bar Z^{\gamma_{t}}(r)) dr \\
   &\ + \int_s^T \bar g(X^{\gamma_{t}}_{r},\bar Y^{\gamma_{t}}(r),\bar Z^{\gamma_{t}}(r)) d\oaB(r) - \int_s^T \bar Z^{\gamma_{t}}(r) dW(r),
\end{align*}
where $\bar f(\gamma_s,y,z)=f(\gamma_s,y-U(s)-V(s),z)$ and $\bar g(\gamma_s,y,z)=g(\gamma_s,y-U(s)-V(s),z)$.
Now, applying Lemma \ref{26}, we conclude that for some constant (we still denoted by $C_p$),
\begin{align*}
  &E\bigg[\sup_{s\in[t,T]}|Y^{\gamma_{t}}(s)|^{p}+\bigg| \int_t^T |Z^{\gamma_{t}}(s)|^{2} ds \bigg|^{\frac{p}{2}}\bigg]\nonumber\\
  &\leq  C_{p} E\bigg[\big|\Phi(X_{T}^{\gamma_{t}})\big|^p
  +\big|U(T)\big|^p+\big|V(T)\big|^p+\int_t^T|f(X_{s}^{\gamma_{t}},0,0)|^pds
  +\int_t^T|g(X_{s}^{\gamma_{t}},0,0)|^pds\bigg]. \label{21.4}
\end{align*}
By the definition of $\bar f$ and $\bar g$, using (H2), we have that
\begin{align*}
  &|\bar f(X_{s}^{\gamma_{t}},0,0)|\leq |f(X_{s}^{\gamma_{t}},0,0)|+C(|U(s)|+|V(s)|), \\
  &|\bar g(X_{s}^{\gamma_{t}},0,0)|\leq |g(X_{s}^{\gamma_{t}},0,0)|+C(|U(s)|+|V(s)|),
\end{align*}
which, combining \rf{19} and Burkholder-Davis-Gundy's inequality, implies that \rf{21.2} hold. This completes the proof.
\end{proof}

Now for the ``parameter'' $\gamma_{t}$, we study the regularities of the solution of BDSDE (\ref{BDSDE}) with respect to it. In the following, for the case of $0\leq s< t\leq T$, we define that $Y^{\gamma_{t}}(s)=Y^{\gamma_{t}}(s\vee t)$ and $Z^{\gamma_{t}}(s)=0$.

\begin{theorem}\label{11} \sl
For any $p\geq 2$, there exist two constants $C_{p}$ and $q$ depending only on $C,T,m, \a$ and $p$ such that
for any $t,\overline{t}\in [0,T]$, $\gamma_{t},\overline{\gamma}_{\overline{t}}\in \Lambda$ and $h,\overline{h}\in \mathbb{R}\setminus \{0\}$,
\begin{align*}
     (1)\q &E\bigg[\sup_{u\in[t\vee \overline{t},T]}|Y^{\gamma_{t}}(u)-Y^{\overline{\gamma}_{\overline{t}}}(u)|^{p}\bigg]
      \leq C_{p}\big(1+\|\gamma_{t}\|^{q}+\|\overline{\gamma}_{\overline{t}}\|^{q}\big) d_{\infty}^{p}(\gamma_{t},\overline{\gamma}_{\overline{t}});\\
     (2)\q &E\bigg[\bigg| \int_{t\vee \overline{t}}^T |Z^{\gamma_{t}}(u) - Z^{\overline{\gamma}_{\overline{t}}}(u)|^{2} du \bigg|^{\frac{p}{2}}\bigg]
      \leq C_{p}\big(1+\|\gamma_{t}\|^{q}+\|\overline{\gamma}_{\overline{t}}\|^{q}\big) d_{\infty}^{p}(\gamma_{t},\overline{\gamma}_{\overline{t}});\\
     (3)\q &E\bigg[\sup_{u\in[t\vee \overline{t},T]}|\Delta_{h}^{i}Y^{\gamma_{t}}(u)
                        -\Delta_{\overline{h}}^{i}Y^{\overline{\gamma}_{\overline{t}}}(u)|^{p}\bigg]  \nonumber\\
      &\leq C_{p}\big(1+\|\gamma_{t}\|^{q}+\|\overline{\gamma}_{\overline{t}}\|^{q}+|h|^{q}+|\overline{h}|^{q}\big)
      \big(|h|^p{\bf1}_{t<\overline t}+|\overline h|^p{\bf1}_{\overline t<t}+
      |h-\overline{h}|^{p} + d_{\infty}^{p}(\gamma_{t},\overline{\gamma}_{\overline{t}})\big);\\
     (4)\q &E\bigg[\bigg( \int_{t\vee \overline{t}}^T |\Delta_{h}^{i}Z^{\gamma_{t}}(u)
                        - \Delta_{\overline{h}}^{i}Z^{\overline{\gamma}_{\overline{t}}}(u)|^{2} du \bigg)^{\frac{p}{2}}\bigg]  \nonumber \\
      &\leq C_{p}\big(1+\|\gamma_{t}\|^{q}+\|\overline{\gamma}_{\overline{t}}\|^{q}+|h|^{q}+|\overline{h}|^{q}\big)
      \big(|h|^p{\bf1}_{t<\overline t}+|\overline h|^p{\bf1}_{\overline t<t}+
      |h-\overline{h}|^{p} + d_{\infty}^{p}(\gamma_{t},\overline{\gamma}_{\overline{t}})\big),
\end{align*}
where
\begin{equation*}
  \Delta_{h}^{i}Y^{\gamma_{t}}(s)=\frac{1}{h}\big(Y^{\gamma_{t}^{he_{i}}}(s)-Y^{\gamma_{t}}(s)\big), \qq
  \Delta_{h}^{i}Z^{\gamma_{t}}(s)=\frac{1}{h}\big(Z^{\gamma_{t}^{he_{i}}}(s)-Z^{\gamma_{t}}(s)\big),
\end{equation*}
and $(e_{1},\ldots,e_{d})$ is an orthonormal basis of $\mathbb{R}^{d}$.
\end{theorem}

\begin{proof}
Note that $(Y^{\gamma_{t}}-Y^{\overline{\gamma}_{\overline{t}}},Z^{\gamma_{t}}-Z^{\overline{\gamma}_{\overline{t}}})$ can be formed as the following linearized
BDSDE. To see this, we consider for each $t\vee \overline{t} \leq s\leq T$,
\begin{align*}
&Y^{\gamma_{t}}(s)-Y^{\overline{\gamma}_{\overline{t}}}(s)\\
=& \Phi(X^{\gamma_{t}}_{T})-\Phi(X^{\overline{\gamma}_{\overline{t}}}_{T})
 +\int_s^T \big[ f(X_{r}^{\gamma_{t}},Y^{\gamma_{t}}(r),Z^{\gamma_{t}}(r)) -
  f(X_{r}^{\overline{\gamma}_{\overline{t}}},Y^{\overline{\gamma}_{\overline{t}}}(r),
  Z^{\overline{\gamma}_{\overline{t}}}(r)) \big] dr\\
 &+ \int_s^T \big[ g(X_{r}^{\gamma_{t}},Y^{\gamma_{t}}(r),Z^{\gamma_{t}}(r)) -
    g(X_{r}^{\overline{\gamma}_{\overline{t}}},Y^{\overline{\gamma}_{\overline{t}}}(r),
    Z^{\overline{\gamma}_{\overline{t}}}(r)) \big] d\oaB(r)\\
 & - \int_s^T \big[Z^{\gamma_{t}}(r)-Z^{\overline{\gamma}_{\overline{t}}}(r)\big] dW(r)\\
=& \Phi(X^{\gamma_{t}}_{T})-\Phi(X^{\overline{\gamma}_{\overline{t}}}_{T})
-\int_s^T\big[Z^{\gamma_{t}}(r)-Z^{\overline{\gamma}_{\overline{t}}}(r)\big] dW(r)\\
 &+\int_s^T \big[
 \overline{f}_{\gamma_{t},\overline{\gamma}_{\overline{t}}}(r)\big( Y^{\gamma_{t}}(r)-Y^{\overline{\gamma}_{\overline{t}}}(r) \big)
                +\widetilde{f}_{\gamma_{t},\overline{\gamma}_{\overline{t}}}(r)\big( Z^{\gamma_{t}}(r)-Z^{\overline{\gamma}_{\overline{t}}}(r) \big)
                +\widehat{f}_{\gamma_{t},\overline{\gamma}_{\overline{t}}}(r) \big] dr\\
 &+\int_s^T \big[ \overline{g}_{\gamma_{t},\overline{\gamma}_{\overline{t}}}(r)\big( Y^{\gamma_{t}}(r)-Y^{\overline{\gamma}_{\overline{t}}}(r) \big)
                +\widetilde{g}_{\gamma_{t},\overline{\gamma}_{\overline{t}}}(r)\big( Z^{\gamma_{t}}(r)-Z^{\overline{\gamma}_{\overline{t}}}(r) \big)
                +\widehat{g}_{\gamma_{t},\overline{\gamma}_{\overline{t}}}(r) \big] d\oaB(r),
\end{align*}
where (with $U^{\gamma_{t}}=(Y^{\gamma_{t}},Z^{\gamma_{t}}))$,
%%\tb{THE following $\a$ and $\b$ are already used before, change them.}
\begin{align}
\overline{f}_{\gamma_{t},\overline{\gamma}_{\overline{t}}}(r)&=\int_0^1 \frac{\partial f}{\partial y}(X_{r}^{\gamma_{t}},
    U^{\overline{\gamma}_{\overline{t}}}(r)+\theta(U^{\gamma_{t}}(r)-U^{\overline{\gamma}_{\overline{t}}}(r))) d\theta,\nonumber\\
\widetilde{f}_{\gamma_{t},\overline{\gamma}_{\overline{t}}}(r)&=\int_0^1 \frac{\partial f}{\partial z}(X_{r}^{\gamma_{t}},
    U^{\overline{\gamma}_{\overline{t}}}(r)+\theta(U^{\gamma_{t}}(r)-U^{\overline{\gamma}_{\overline{t}}}(r))) d\theta, \label{3.19.2}\\
\widehat{f}_{\gamma_{t},\overline{\gamma}_{\overline{t}}}(r)&=
f(X_{r}^{\gamma_{t}},Y^{\overline{\gamma}_{\overline{t}}}(r),Z^{\overline{\gamma}_{\overline{t}}}(r))
    -f(X_{r}^{\overline{\gamma}_{\overline{t}}},
    Y^{\overline{\gamma}_{\overline{t}}}(r),Z^{\overline{\gamma}_{\overline{t}}}(r)),\nonumber
\end{align}
and $\overline{g}_{\gamma_{t},\overline{\gamma}_{\overline{t}}},\
\widetilde{g}_{\gamma_{t},\overline{\gamma}_{\overline{t}}}$ and
$\widehat{g}_{\gamma_{t},\overline{\gamma}_{\overline{t}}}$ are defined analogously with $f$ replaced by $g$. Now, from (H3) and (H4), we have that
\begin{align*}
  &|\widehat{f}_{\gamma_{t},\overline{\gamma}_{\overline{t}}}|
  +|\widehat{g}_{\gamma_{t},\overline{\gamma}_{\overline{t}}}|
  +|\Phi(X^{\gamma_{t}}_{T})-\Phi(X^{\overline{\gamma}_{\overline{t}}}_{T})|\\
  &\leq 3C\big(1+\|X_{T}^{\gamma_{t}}\|^m+\|X_{T}^{\overline{\gamma}_{\overline{t}}}\|^m\big)
  \|X_{T}^{\gamma_{t}}-X_{T}^{\overline{\gamma}_{\overline{t}}}\|^m.
\end{align*}
Then, from Lemma \ref{26.2} and Burkholder-Davis-Gundy's inequality, note that \rf{1.2.3} and \rf{19},  we obtain that the first two inequalities hold.

\ms

For the last two inequalities, similarly, we can write $(\Delta_{h}^{i}Y^{\gamma_{t}},\Delta_{h}^{i}Z^{\gamma_{t}})$ as the solution of the following linearized BDSDE:
\begin{align*}
\Delta_{h}^{i}Y^{\gamma_{t}}(s)
=& \frac{1}{h}\big[\Phi(X^{\gamma_{t}^{he_{i}}}_{T})-\Phi(X^{\gamma_{t}}_{T})\big]
  -\int_s^T \Delta_{h}^{i}Z^{\gamma_{t}}(r) dW(r)\\
 &+\int_s^T \big[ \overline{f}_{\gamma_{t}^{he_{i}},\gamma_{t}}(r) \Delta_{h}^{i}Y^{\gamma_{t}}(r)
                 +\widetilde{f}_{\gamma_{t}^{he_{i}},\gamma_{t}}(r)\Delta_{h}^{i}Z^{\gamma_{t}}(r)
                 +\frac{1}{h}\widehat{f}_{\gamma_{t}^{he_{i}},\gamma_{t}}(r) \big] dr\\
 &+\int_s^T \big[ \overline{g}_{\gamma_{t}^{he_{i}},\gamma_{t}}(r) \Delta_{h}^{i}Y^{\gamma_{t}}(r)
                 +\widetilde{g}_{\gamma_{t}^{he_{i}},\gamma_{t}}(r)\Delta_{h}^{i}Z^{\gamma_{t}}(r)
                 +\frac{1}{h}\widehat{g}_{\gamma_{t}^{he_{i}},\gamma_{t}}(r) \big] d\oaB(r),
\end{align*}
where $(e_{1},\ldots,e_{d})$ is an orthonormal basis of $\mathbb{R}^{d}$, and $\overline{f}, \widetilde{f},\widehat{f}$ and $\overline{g}, \widetilde{g},\widehat{g}$ are defined similarly as in \rf{3.19.2}. Then the same calculus as the above implies that
\begin{align*}
  E\bigg[\sup_{s\in[t,T]}|\Delta_{h}^{i}Y^{\gamma_{t}}(s)|^{p}+\bigg| \int_t^T |\Delta_{h}^{i}Z^{\gamma_{t}}(s)|^{2} ds \bigg|^{\frac{p}{2}}\bigg]
  \leq  C_{p} \big(1+|h|^{q}+\| \gamma_{t}\|^{q}\big).
\end{align*}
Finally, we consider that
\begin{align*}
&\Delta_{h}^{i}Y^{\gamma_{t}}(s) - \Delta_{\overline{h}}^{i}Y^{\overline{\gamma}_{\overline{t}}}(s)\\
=& \frac{1}{h}\big[\Phi(X^{\gamma_{t}^{he_{i}}}_{T})-\Phi(X^{\gamma_{t}}_{T})\big]
 - \frac{1}{\bar{h}}\big[\Phi(X^{\overline{\gamma}_{\overline{t}}^{\overline{h}e_{i}}}_{T})
 -\Phi(X^{\overline{\gamma}_{\overline{t}}}_{T})\big]\\
 &+\int_s^T  \big[\overline{f}_{\gamma_{t}^{he_{i}},\gamma_{t}}(r) \Delta_{h}^{i}Y^{\gamma_{t}}(r)
  -\overline{f}_{\overline{\gamma}_{\overline{t}}^{\overline{h}e_{i}},\overline{\gamma}_{\overline{t}}}(r) \Delta_{\overline{h}}^{i}Y^{\overline{\gamma}_{\overline{t}}}(r)\big]dr\\
 &+\int_s^T  \big[\widetilde{f}_{\gamma_{t}^{he_{i}},\gamma_{t}}(r) \Delta_{h}^{i}Z^{\gamma_{t}}(r)
  -\widetilde{f}_{\overline{\gamma}_{\overline{t}}^{\overline{h}e_{i}},\overline{\gamma}_{\overline{t}}}(r) \Delta_{\overline{h}}^{i}Z^{\overline{\gamma}_{\overline{t}}}(r)\big]dr\\
 &+\int_s^T  \big[\frac{1}{h}\widehat{f}_{\gamma_{t}^{he_{i}},\gamma_{t}}(r)
                  -\frac{1}{\bar{h}}\widehat{f}_{\overline{\gamma}_{\overline{t}}^{\overline{h}e_{i}},\overline{\gamma}_{\overline{t}}}(r) \big] dr \\
 &+\int_s^T  \big[\overline{g}_{\gamma_{t}^{he_{i}},\gamma_{t}}(r) \Delta_{h}^{i}Y^{\gamma_{t}}(r)
  -\overline{g}_{\overline{\gamma}_{\overline{t}}^{\overline{h}e_{i}},\overline{\gamma}_{\overline{t}}}(r) \Delta_{\overline{h}}^{i}Y^{\overline{\gamma}_{\overline{t}}}(r)\big]d\oaB(r)\\
 &+\int_s^T  \big[\widetilde{g}_{\gamma_{t}^{he_{i}},\gamma_{t}}(r) \Delta_{h}^{i}Z^{\gamma_{t}}(r)
  -\widetilde{g}_{\overline{\gamma}_{\overline{t}}^{\overline{h}e_{i}},\overline{\gamma}_{\overline{t}}}(r) \Delta_{\overline{h}}^{i}Z^{\overline{\gamma}_{\overline{t}}}(r)\big]d\oaB(r)\\
 &+\int_s^T  \big[\frac{1}{h}\widehat{g}_{\gamma_{t}^{he_{i}},\gamma_{t}}(r)
                  -\frac{1}{\bar{h}}\widehat{g}_{\overline{\gamma}_{\overline{t}}^{\overline{h}e_{i}},\overline{\gamma}_{\overline{t}}}(r) \big] d\oaB(r)\\
 &-\int_s^T [\Delta_{h}^{i}Z^{\gamma_{t}}(r)
 -\Delta_{\overline{h}}^{i}Z^{\overline{\gamma}_{\overline{t}}}(r)] dW(r),\qq r\in[t\vee\overline{t},T].
\end{align*}
Thus $(\widetilde{Y}(s),\widetilde{Z}(s))\deq (\Delta_{h}^{i}Y^{\gamma_{t}}(s)-\Delta_{\overline{h}}^{i}Y^{\overline{\gamma}_{\overline{t}}}(s),
\Delta_{h}^{i}Z^{\gamma_{t}}(s)-\Delta_{\overline{h}}^{i}Z^{\overline{\gamma}_{\overline{t}}}(s))$ solves the following linearized BDSDE:
\begin{align*}
\widetilde{Y}(s)
=&  \frac{1}{h}\big[\Phi(X^{\gamma_{t}^{he_{i}}}_{T})-\Phi(X^{\gamma_{t}}_{T})\big]
 - \frac{1}{\bar{h}}\big[\Phi(X^{\overline{\gamma}_{\overline{t}}^{\overline{h}e_{i}}}_{T})-\Phi(X^{\overline{\gamma}_{\overline{t}}}_{T})\big]
  -\int_s^T \widetilde{Z}(r) dW(r)\\
&+\int_s^T \big[\overline{f}_{\gamma_{t}^{he_{i}},\gamma_{t}}(r)\widetilde{Y}(r)
 +\widetilde{f}_{\gamma_{t}^{he_{i}},\gamma_{t}}(r)\widetilde{Z}(r)+\breve{f}(r)\big] dr\\
&+\int_s^T \big[\overline{g}_{\gamma_{t}^{he_{i}},\gamma_{t}}(r)\widetilde{Y}(r)
 +\widetilde{g}_{\gamma_{t}^{he_{i}},\gamma_{t}}(r)\widetilde{Z}(r)+\breve{g}(r)\big] d\oaB(r),
\end{align*}
where
\begin{align*}
\breve{f}(r)\deq  &
\big[\overline{f}_{\gamma_{t}^{he_{i}},\gamma_{t}}(r)
-\overline{f}_{\overline{\gamma}_{\overline{t}}^{\overline{h}e_{i}},
\overline{\gamma}_{\overline{t}}}(r)\big]
\Delta_{\overline{h}}^{i}Y^{\overline{\gamma}_{\overline{t}}}(r)
+ \big[\widetilde{f}_{\gamma_{t}^{he_{i}},\gamma_{t}}(r)
-\widetilde{f}_{\overline{\gamma}_{\overline{t}}^{\overline{h}e_{i}},\overline{\gamma}_{\overline{t}}}(r)\big] \Delta_{\overline{h}}^{i}Z^{\overline{\gamma}_{\overline{t}}}(r)\\
&+ \big[\frac{1}{h}\widehat{f}_{\gamma_{t}^{he_{i}},\gamma_{t}}(r)
  -\frac{1}{\bar{h}}\widehat{f}_{\overline{\gamma}_{\overline{t}}^{\overline{h}e_{i}},\overline{\gamma}_{\overline{t}}}(r) \big],
\end{align*}
and $\breve{g}(r)$ is defined analogously. By (H3) and (H4), we have that
\begin{align*}
  &\Phi(X^{\gamma_{t}^{he_{i}}}_{T})-\Phi(X^{\gamma_{t}}_{T})
  =h\int_0^1\big\langle\Phi'_{\gamma_t}(X^{\gamma_{t}^{\theta he_{i}}}_{T}),e_i\big\rangle d\theta,\\
 & \widehat{f}_{\gamma_{t}^{he_{i}},\gamma_{t}}(r)
  =h\int_0^1\big\langle f'_{\gamma_t}\big(X^{\gamma_{t}^{\theta he_{i}}}_{r},
  Y^{\gamma_t}(r),Z^{\gamma_t}(r)\big),e_i\big\rangle d\theta,\\
\end{align*}
and $\widehat{g}_{\gamma_{t}^{he_{i}},\gamma_{t}}(r)$ is analogous.
So there exists a constant $C'_p$, which depending only on $C, T, m, \a$ and $p$, and note that \rf{1.2.3} and \rf{19}, such that
\begin{align*}
  &E\bigg[\bigg| \frac{1}{h}\big[\Phi(X^{\gamma_{t}^{he_{i}}}_{T})-\Phi(X^{\gamma_{t}}_{T})\big]
  -\frac{1}{\bar{h}}\big[\Phi(X^{\overline{\gamma}_{\overline{t}}^{\overline{h}e_{i}}}_{T})
  -\Phi(X^{\overline{\gamma}_{\overline{t}}}_{T})\big] \bigg|^p\bigg]\\
  &=\bigg[\bigg| \int_0^1\big\langle\Phi'_{\gamma_t}(X^{\gamma_{t}^{\theta he_{i}}}_{T}),e_i\big\rangle d\theta - \int_0^1\big\langle\Phi'_{\overline{\gamma}_{\overline{t}}}
  (X^{\overline{\gamma}_{\overline{t}}^{\theta \overline{h}e_{i}}}_{T}),
  e_i\big\rangle d\theta \bigg|^p\bigg]\\
&\leq C'_{p}\big(1+\|\gamma_{t}\|^{q}+\|\overline{\gamma}_{\overline{t}}\|^{q}+|h|^{q}+|\overline{h}|^{q}\big)
      \big(|h|^p{\bf1}_{t<\overline t}+|\overline h|^p{\bf1}_{\overline t<t}+
      |h-\overline{h}|^{p} + d_{\infty}^{p}(\gamma_{t},\overline{\gamma}_{\overline{t}})\big).
\end{align*}
Furthermore, we also have
\begin{align*}
&E\bigg[\bigg|\int_t^T|\breve{f}(r)|dr\bigg|^p\bigg]\\
&\leq C'_{p}\big(1+\|\gamma_{t}\|^{q}+\|\overline{\gamma}_{\overline{t}}\|^{q}
+|h|^{q}+|\overline{h}|^{q}\big)\cdot\big(|h|^p{\bf1}_{t<\overline t}
+|\overline h|^p{\bf1}_{\overline t<t}+|h-\overline{h}|^{p} + d_{\infty}^{p}(\gamma_{t},\overline{\gamma}_{\overline{t}})\big),
\end{align*}
and a similar result for $\breve{g}$ could be obtained by combining Burkholder-Davis-Gundy's inequality. Finally, using Lemma \ref{26.2} once again, we can get the last two inequalities hold.
\end{proof}

Based on the above theorem, we could obtain the following lemma concerning $\{Y^{\gamma_{t}^{x}}(\cd)$.

\begin{lemma}\label{12} \sl
For every element $\gamma_{t}\in \Lambda$, $\{Y^{\gamma_{t}^{x}}(s), s\in [0,T], x\in \mathbb{R}^{d}\}$ has a version which is $a.e.$ of class $\dbC^{0,2}([0,T]\times \mathbb{R}^{d})$.
\end{lemma}

\begin{proof}
%We only consider one dimensional case, as the higher dimensional case can be treated in the same way without difficulty.
%
%\ms
%
Applying the Kolmogorov's criterion, it is easy to see that Theorem \ref{11} implies the existence of a continuous derivative of $Y^{\gamma_{t}^{x}}(s)$ with respect to $x$, and the existence of a mean square derivative of $Z^{\gamma_{t}^{x}}(s)$ with respect to $x$, which is mean square continuous in $x$.
Denote them by $(D_{x}Y^{\gamma_{t}},D_{x}Z^{\gamma_{t}})$. From the proof of Theorem \ref{11}, we see that  $(D_{x}Y^{\gamma_{t}},D_{x}Z^{\gamma_{t}})$ is the solution of the following linearized BDSDE:
\begin{align*}
D_{x}Y^{\gamma_{t}}(s)=&\ \Phi'_{t}(X_{T}^{\gamma_{t}}) - \int_s^T D_{x}Z^{\gamma_{t}}(r) dW(r)\\
 &+ \int_s^T \big[ f_{y}'(X_{r}^{\gamma_{t}},Y^{\gamma_{t}}(r),Z^{\gamma_{t}}(r))D_{x}Y^{\gamma_{t}}(r)
                 +f_{z}'(X_{r}^{\gamma_{t}},Y^{\gamma_{t}}(r),Z^{\gamma_{t}}(r))D_{x}Z^{\gamma_{t}}(r)\\
   &\ \ \ \ \ \ \ \ \ \ +f_{t}'(X_{r}^{\gamma_{t}},Y^{\gamma_{t}}(r),Z^{\gamma_{t}}(r)) \big] dr\\
 &+ \int_s^T \big[ g_{y}'(X_{r}^{\gamma_{t}},Y^{\gamma_{t}}(r),Z^{\gamma_{t}}(r))D_{x}Y^{\gamma_{t}}(r)
                 +g_{z}'(X_{r}^{\gamma_{t}},Y^{\gamma_{t}}(r),Z^{\gamma_{t}}(r))D_{x}Z^{\gamma_{t}}(r)\\
   &\ \ \ \ \ \ \ \ \ \ +g_{t}'(X_{r}^{\gamma_{t}},Y^{\gamma_{t}}(r),Z^{\gamma_{t}}(r))\big]
   d\oaB(r).
\end{align*}
One can check that the above BDSDE has a unique solution. From Corollary \ref{3.20.1} (see below), we conclude that $\sup_{s\in[t,T]}|Z^{\gamma^k_t}(s)|$ belongs to the space $L^{p}(\Omega,\mathcal{F}_{T},\dbP)$ for each $\geq 2$.
Thus using a similar fashion, the existence of a continuous second order partial derivative of $Y^{\gamma_{t}^{x}}(s)$ with respect to $x$ can be proved.
\end{proof}

Now we define
\begin{equation}\label{27}
  u(\gamma_{t})\deq Y^{\gamma_{t}}(t), \q~ \text{for} \q~ \gamma_{t}\in \Lambda.
\end{equation}

\begin{remark} \rm
It should be noticed that $u(\gamma_{t})$ defined in (\ref{27}) is an $\mathcal{F}_{t,T}^{B}$-measurable random field for each $\gamma_{t}\in \Lambda$,
because of the emergence of the backward It\^{o}'s integral term ``$dB$'' in (\ref{BDSDE}).
In fact, this is one of the essential difference between BSDEs and BDSDEs.
\end{remark}

We have the following results about $u(\gamma_{t})$.

\begin{lemma}\label{16} \sl
For any $\bar{t}\in[t,T]$, one has $u(X_{\bar{t}}^{\gamma_{t}})=Y^{\gamma_{t}}(\bar{t})$, a.s.
\end{lemma}

\begin{proof}
For any given $\gamma_{t}\in\Lambda$, $t\leq \bar{t},$  set $X_{t}=\gamma_{t}$.
Consider the solution of SDE (\ref{SDE}) and BDSDE (\ref{BDSDE}) on $[t,T]$:
\begin{equation*}
\begin{split}
X^{\gamma_{t}}(s)=&X^{\gamma_{t}}(\bar{t})+\int_{\bar{t}}^s b(X_{r}^{\gamma_{t}}) dr+\int_{\bar{t}}^s \sigma(X_{r}^{\gamma_{t}}) dW(r),\\
Y^{\gamma_{t}}(s)=&\Phi(X_{T}^{\gamma_{t}}) + \int_s^T f(X_{r}^{\gamma_{t}},Y^{\gamma_{t}}(r),Z^{\gamma_{t}}(r)) dr\\
   &+ \int_s^T g(X^{\gamma_{t}}_{r},Y^{\gamma_{t}}(r),Z^{\gamma_{t}}(r)) d\oaB(r) - \int_s^T Z^{\gamma_{t}}(r) dW(r).
\end{split}
\end{equation*}
We need to prove $u(X_{\bar{t}}^{\gamma_{t}})=Y^{\gamma_{t}}(\bar{t})$.
For each fixed $s\in [t,T]$ and positive integer $n$,
we introduce a mapping $\gamma^{n}(\overline{\gamma}_{s}):\Lambda_{s}\mapsto \Lambda_{s}$,
\begin{equation*}
   \gamma^{n}(\overline{\gamma}_{s})(r)
  =\overline{\gamma}_{s}(r)I_{[0,t)}(r)
   +\sum_{i=0}^{n-1}\overline{\gamma}_{s}(t_{i}^{n}\wedge s)I_{[t_{i}^{n}\wedge s,t_{i}^{n+1}\wedge s)}(r)
   +\overline{\gamma}_{s}(s)I_{\{s\}}(r),\q~0\leq r\leq s,
\end{equation*}
where $t_{i}^{n}=t+\frac{i(T-t)}{n}, i=0, 1, \ldots,n$.
Denote
\begin{align*}
\gamma^{n} (X_{\bar{t}}^{\gamma_{t}})(r)=X^{n,\gamma_{t}}(r), \q~  t\leq r\leq \bar{t},
\end{align*}
and set
\begin{align*}
X^{n,\gamma_{t}}_{\bar{t}}\deq \sum_{j=1}^{N}I_{A_{j}}x_{\bar{t}}^{j},
\end{align*}
where $\{A_{j}\}_{j=1}^{N}$ is a division of $\mathcal{F}_{\bar{t}}$, and $x_{\bar{t}}^{j}\in \Lambda_{\bar{t}}$, $j=1,2,\ldots, N$.
Then, for any $j$, $(Y^{x_{\bar{t}}^{j}}(s), Z^{x_{\bar{t}}^{j}}(s))$ is the solution of the following BDSDE: for $s\in [\bar{t},T]$,
\begin{align*}
Y^{x_{\bar{t}}^{j}}(s)=&\Phi(X_{T}^{x_{\bar{t}}^{j}}) + \int_s^T f(X_{r}^{x_{\bar{t}}^{j}},Y^{x_{\bar{t}}^{j}}(r),Z^{x_{\bar{t}}^{j}}(r)) dr\\
   &+ \int_s^T g(X_{r}^{x_{\bar{t}}^{j}},Y^{x_{\bar{t}}^{j}}(r),Z^{x_{\bar{t}}^{j}}(r)) d\oaB(r)
   - \int_s^T Z^{x_{\bar{t}}^{j}}(r) dW(r).
\end{align*}
Multiplying by $I_{A_{j}}$ and adding the corresponding terms, we obtain
\begin{align*}
\sum_{j=1}^{N}I_{A_{j}}Y^{x_{\bar{t}}^{j}}(s)
=&\Phi(\sum_{j=1}^{N}I_{A_{j}}X_{T}^{x_{\bar{t}}^{j}}) - \int_s^T \sum_{j=1}^{N}I_{A_{j}}Z^{x_{\bar{t}}^{j}}(r) dW(r)\\
 &+ \int_s^T f(\sum_{j=1}^{N}I_{A_{j}}X_{r}^{x_{\bar{t}}^{j}},
 \sum_{j=1}^{N}I_{A_{j}}Y^{x_{\bar{t}}^{j}}(r),\sum_{j=1}^{N}I_{A_{j}}Z^{x_{\bar{t}}^{j}}(r)) dr\\
&+\int_s^T g(\sum_{j=1}^{N}I_{A_{j}}X_{r}^{x_{\bar{t}}^{j}},
\sum_{j=1}^{N}I_{A_{j}}Y^{x_{\bar{t}}^{j}}(r),\sum_{j=1}^{N}I_{A_{j}}Z^{x_{\bar{t}}^{j}}(r))d\oaB(r).
\end{align*}
From the existence and uniqueness theorem of BDSDE, we have
\begin{align*}
Y^{X^{n,\gamma_{t}}_{\bar{t}}}(s) = \sum_{j=1}^{N}I_{A_{j}}Y^{x_{\bar{t}}^{j}}(s), \q~
Z^{X^{n,\gamma_{t}}_{\bar{t}}}(s) = \sum_{j=1}^{N}I_{A_{j}}Z^{x_{\bar{t}}^{j}}(s), \q a.s., \ s\in[\bar{t},T].
\end{align*}
Then, by the definition of $u$, we get
\begin{align*}
Y^{X^{n,\gamma_{t}}_{\bar{t}}}(\bar{t})
= \sum_{j=1}^{N}I_{A_{j}}Y^{x_{\bar{t}}^{j}}(\bar{t})
= \sum_{j=1}^{N}I_{A_{j}}u(x_{\bar{t}}^{j})
= u(X^{n,\gamma_{t}}_{\bar{t}}), \q a.s.
\end{align*}
Note that
\begin{align*}
\lim_{n\rightarrow \infty}X_{\bar{t}}^{n,\gamma_{t}}=X_{\bar{t}}^{\gamma_{t}}, \q a.s.
\end{align*}
Hence the desired result is obtained.
\end{proof}

Now by the definition of vertical derivative and Theorem \ref{11} and Lemma \ref{12}, moreover, note that \rf{27}, we have the following corollary.

\begin{corollary}\label{13} \sl
For each $\gamma_{t}\in \Lambda$,
the Dupire's path derivatives $D_{x}u(\gamma_{t})$ and $D_{xx}u(\gamma_{t})$ exist a.s., and $u$ belongs to the space $\mathbb{C}^{0,2}_{l,\text{Lip}}(\Lambda)$.
\end{corollary}

\subsection{Path regularity of process Z}

In Pardoux and Peng \cite{Peng94}, when the coefficients of SDE and BDSDE are state-dependent, i.e.,
the functions $\Phi=\Phi(\gamma(T))$, $b=b(t,\gamma(t))$, $\sigma=\sigma(t,\gamma(t))$, $f=f(t,\gamma(t),y,z)$ and $g=g(t,\gamma(t),y,z)$,
it is shown that, under appropriate assumptions, $Y$ and $Z$ are connected in the following sense:
\begin{equation*}
Z^{\gamma_{t}}(s)=\sigma(s,X^{\gamma_{t}}(s)) \partial_{x}u(s,X^{\gamma_{t}}(s)), \q  a.s.
\end{equation*}
In this subsection, we shall extend this result to the case of path-dependence.
Indeed, we have below a formula relating $Z$ with $Y$.

\begin{theorem}\label{17} \sl
For each fixed $\gamma_{t}\in \Lambda$, the process $(Z^{\gamma_{t}}(s))_{s\in [t,T]}$ has an a.s. continuous version with the form
\begin{equation*}
Z^{\gamma_{t}}(s)=\sigma(X^{\gamma_{t}}_{s}) D_{x}u(X^{\gamma_{t}}_{s}), \q~ \text{for all}\ \ s\in [t,T].
\end{equation*}
\end{theorem}

To prove the above theorem, we need the following lemma which is essentially from Pardoux and Peng \cite{Peng92} (see Lemma 2.5 of \cite{Peng92}).

\begin{lemma}\label{14} \sl
Let $\gamma_{t}\in \Lambda$ and some $\overline{t}\in [t,T]$ be given, suppose
\begin{align*}
\Phi(\gamma)&=\varphi(\gamma(\overline{t}),\gamma(T)-\gamma(\overline{t})),
   \ \ \text{where} \ \ \varphi\in \dbC_{l,\text{Lip}}^{2,0}(\mathbb{R}^{d}\times \mathbb{R}^{d};\mathbb{R}^{k}), \\
b(\gamma_{s})&=b_{1}(s,\gamma_{s}(s))I_{[0,\overline{t})}(s) + b_{2}(s,\gamma_{s}(s)-\gamma_{s}(\overline{t}))I_{[\overline{t},T]}(s),   \\
\sigma(\gamma_{s})&=\sigma_{1}(s,\gamma_{s}(s))I_{[0,\overline{t})}(s) + \sigma_{2}(s,\gamma_{s}(s)-\gamma_{s}(\overline{t}))I_{[\overline{t},T]}(s),   \\
f(\gamma_{s},y,z)&=f_{1}(s,\gamma_{s}(s),y,z)I_{[0,\overline{t})}(s)
   + f_{2}(s,\gamma_{s}(\overline{t}),\gamma_{s}(s)-\gamma_{s}(\overline{t}),y,z)I_{[\overline{t},T]}(s),   \\
g(\gamma_{s},y,z)&=g_{1}(s,\gamma_{s}(s),y,z)I_{[0,\overline{t})}(s)
   + g_{2}(s,\gamma_{s}(\overline{t}),\gamma_{s}(s)-\gamma_{s}(\overline{t}),y,z)I_{[\overline{t},T]}(s),
\end{align*}
where  $f_{1}$ and $g_{1}$ satisfy (H5), and $b_{1},b_{2},\sigma_{1},\sigma_{2},f_{2},g_{2}\in \dbC_{l,\text{Lip}}^{0,2}.$ Then for each $s\in[t,T]$,
\begin{equation*}
Z^{\gamma_{t}}(s)=\sigma(X^{\gamma_{t}}_{s}) D_{x}u(X^{\gamma_{t}}_{s}),  \q~ a.s.
\end{equation*}
\end{lemma}

\begin{proof}
We still only consider the one dimensional case.
For $s\in [\overline{t},T]$, SDE (\ref{SDE}) and BDSDE (\ref{BDSDE}) can be rewritten as
\begin{align*}
  X^{\gamma_{s}}(u)=&\gamma_{s}(s)+\int_s^u b_{2}(r,X^{\gamma_{s}}(r)-\gamma_{s}(\overline{t})) dr
                     +\int_s^u \sigma_{2}(r,X^{\gamma_{s}}(r)-\gamma_{s}(\overline{t})) dW(r),\\
  Y^{\gamma_{s}}(u)=&\varphi(\gamma_{s}(\overline{t}),X^{\gamma_{s}}(T)-\gamma_{s}(\overline{t}))
  +\int_u^Tf_{2}(r,\gamma_{s}(\overline{t}),X^{\gamma_{s}}(r)
  -\gamma_{s}(\overline{t}),Y^{\gamma_{s}}(r),Z^{\gamma_{s}}(r)) dr\\
  &+\int_u^Tg_{2}(r,\gamma_{s}(\overline{t}),X^{\gamma_{s}}(r)
  -\gamma_{s}(\overline{t}),Y^{\gamma_{s}}(r),Z^{\gamma_{s}}(r)) d\oaB(r)
                     -\int_u^T Z^{\gamma_{s}}(r) dW(r), \q~ u\in[s,T].
\end{align*}
Similarly, for $s\in [t,\overline{t}]$,
\begin{align*}
 X^{\gamma_{s}}(u)=&\gamma_{s}(s)+\int_s^u b_{1}(r,X^{\gamma_{s}}(r))dr+\int_s^u \sigma_{1}(r,X^{\gamma_{s}}(r)) dW(r), \ u\in [s,\overline{t}],\\
 X^{\gamma_{s}}(u)=&X^{\gamma_{s}}(\overline{t})+\int_{\overline{t}}^u b_{2}(r,X^{\gamma_{s}}(r)-X^{\gamma_{s}}(\overline{t})) dr
    +\int_{\overline{t}}^u \sigma_{2}(r,X^{\gamma_{s}}(r)-X^{\gamma_{s}}(\overline{t})) dW(r), \q~ u\in [\overline{t},T],
\end{align*}
\begin{align*}
 Y^{\gamma_{s}}(u)=&\varphi( X^{\gamma_{s}}(\overline{t}),X^{\gamma_{s}}(T)- X^{\gamma_{s}}(\overline{t}))
    +\int_u^T f_{2}(r, X^{\gamma_{s}}(\overline{t}),X^{\gamma_{s}}(r)- X^{\gamma_{s}}(\overline{t}),Y^{\gamma_{s}}(r),Z^{\gamma_{s}}(r)) dr\\
   &+\int_u^T g_{2}(r, X^{\gamma_{s}}(\overline{t}),X^{\gamma_{s}}(r)- X^{\gamma_{s}}(\overline{t}),Y^{\gamma_{s}}(r),Z^{\gamma_{s}}(r)) d\oaB(r)
    -\int_u^T Z^{\gamma_{s}}(r) dW(r), \q~ u\in [\overline{t},T],\\
Y^{\gamma_{s}}(u)=&Y^{\gamma_{s}}(\overline{t})+\int_u^{\overline{t}} f_{1}(r,X^{\gamma_{s}}(r),Y^{\gamma_{s}}(r),Z^{\gamma_{s}}(r)) dr\\
   &+\int_u^{\overline{t}} g_{1}(r,X^{\gamma_{s}}(r),Y^{\gamma_{s}}(r),Z^{\gamma_{s}}(r)) d\oaB(r)-\int_u^{\overline{t}} Z^{\gamma_{s}}(r) dW(r),
    \q~ u\in [s,\overline{t}].
\end{align*}
Now consider the following system of semi-linear parabolic stochastic partial differential equation
defined on $[\overline{t},T]\times \mathbb{R}\times \mathbb{R}$ and parameterized by $x\in \mathbb{R}$,
\begin{align*}
  u_{2}(s,x,y) =& \varphi (x,y) + \int_s^T \big[ \mathcal{L}u_{2}(r,x,y) + f_{2}(r,x,y,u_{2}(r,x,y),\sigma_{2}(r,y)\partial_{y}u_{2}(r,x,y)) \big] dr\\
                & + \int_s^T g_{2}(r,x,y,u_{2}(r,x,y),\sigma_{2}(r,y)\partial_{y}u_{2}(r,x,y))  d\oaB(r),
\end{align*}
where
\begin{equation*}
  \mathcal{L}u_{2}(r,x,y) = b_{2}(r,y)\partial_{y}u_{2}(r,x,y) + \frac{1}{2}\sigma_{2}^{2}(r,y)\partial^{2}_{yy}u_{2}(r,x,y).
\end{equation*}
Recalling Theorem 3.1 and 3.2 in Pardoux and Peng \cite{Peng94},
for each $s\in [\overline{t},T]$, $u_{2}(s,\cdot)$ is of class $\dbC_{l,\text{Lip}}^{0,2}(\mathbb{R}\times \mathbb{R};\mathbb{R}),$ a.s..
Similarly, the following semi-linear parabolic stochastic partial differential equation defined on $[t,\overline{t}]\times \mathbb{R}$ also has a unique
solution $u_{1}\in \dbC_{l,\text{Lip}}^{0,2}(\mathbb{R}\times \mathbb{R};\mathbb{R})$ a.s.,
\begin{align*}
  u_{1}(s,x) =&u_{2}(\overline{t},x,0)
    + \int_s^{\overline{t}} \big[ \mathcal{L}u_{1}(r,x) + f_{1}(r,x,u_{1}(r,x),\sigma_{1}(r,x)\partial_{x}u_{1}(r,x)) \big] dr\\
   &+ \int_s^{\overline{t}} g_{1}(r,x,u_{1}(r,x),\sigma_{1}(r,x)\partial_{x}u_{1}(r,x)) d\oaB(r),
\end{align*}
where
\begin{equation*}
  \mathcal{L}u_{1}(r,x) = b_{1}(r,x)\partial_{x}u_{1}(r,x) + \frac{1}{2}\sigma_{1}^{2}(r,x)\partial^{2}_{xx}u_{1}(r,x).
\end{equation*}
Then we get a.s.
\begin{equation*}
\begin{split}
Y^{\gamma_{t}}(s)=&u_{1}(s,X^{\gamma_{t}}(s)),\q~ t\leq s< \overline{t},\\
  Y^{\gamma_{t}}(s)=&u_{2}(s,X^{\gamma_{t}}(\overline{t}),X^{\gamma_{t}}(s)-X^{\gamma_{t}}(\overline{t})), \q~  \overline{t}\leq s\leq T,\\
  Z^{\gamma_{t}}(s)=&\sigma_{1}(s,X^{\gamma_{t}}(s)) \partial_{x}u_{1}(s,X^{\gamma_{t}}(s)), \q~  t\leq s< \overline{t},\\
  Z^{\gamma_{t}}(s)=&\sigma_{2}(s,X^{\gamma_{t}}(s)-X^{\gamma_{t}}(\overline{t}))\cdot
                     \partial_{x}u_{2}(s,X^{\gamma_{t}}(\overline{t}),X^{\gamma_{t}}(s)-X^{\gamma_{t}}(\overline{t})),  \q~  \overline{t}\leq s\leq T.
\end{split}
\end{equation*}
Indeed, one can directly check it by the It\^{o}'s formula and the uniqueness of the solution of BDSDE.
Therefore by the definition of $u$, we have
\begin{equation*}
   u(\gamma_{s})=u_{1}(s,\gamma_{s}(s))I_{[t,\overline{t})}(s)\\
   +u_{2}(s,\gamma_{s}(\overline{t}),\gamma_{s}(s)-\gamma_{s}(\overline{t}))I_{[\overline{t},T]}(s).
\end{equation*}
Finally, from the definition of vertical derivative, for each $s\in[t,T]$,
\begin{equation*}
  Z^{\gamma_{t}}(s)=\sigma(s,X_{s}^{\gamma_{t}})D_{x}u(X_{s}^{\gamma_{t}}), \q~  a.s.
\end{equation*}
Therefore, our conclusion follows.
\end{proof}

\begin{remark}\label{29} \rm
Under the same conditions as Lemma \ref{14}, $u(\gamma_s)$ satisfies the following equation:
\begin{align}\label{15}
 u(\gamma_{t})=&\   \Phi(\gamma_{T}) + \int_t^T \bigg[ \frac{1}{2}tr\big[\big((\sigma\sigma^\top D_{xx}u\big)(\gamma_{s})\big]
 + \langle b(\gamma_{s}),D_{x}u(\gamma_{s}) \rangle + f(\gamma_{s},u(\gamma_{s}),(\sigma D_{x}u)(\gamma_{s})) \bigg] ds \nonumber\\
               &\ + \int_t^T g(\gamma_{s},u(\gamma_{s}),(\sigma D_{x}u)(\gamma_{s})) d\oaB(s),
                \q~ \gamma \in \Lambda, \ \ t\in [0, T],
\end{align}
where
$u(\gamma)=\varphi(\gamma(\overline{t}),\gamma(T)-\gamma(\overline{t}))$.
In particular, equation (\ref{15}) can be regarded as a discrete-type of path-dependent SPDE (\ref{SPDE}).
\end{remark}

\begin{proof}[Proof of Theorem \ref{17}]
For each fixed $s\in [t,T]$ and positive integer $n$,
we introduce a mapping $\gamma^{n}(\overline{\gamma}_{s}):\Lambda_{s}\mapsto \Lambda_{s}$,
\begin{equation*}
   \gamma^{n}(\overline{\gamma}_{s})(r)
  =\overline{\gamma}_{s}(r)I_{[0,t)}(r)
   +\sum_{i=0}^{n-1}\overline{\gamma}_{s}(t_{i}^{n}\wedge s)I_{[t_{i}^{n}\wedge s,t_{i}^{n+1}\wedge s)}(r)
   +\overline{\gamma}_{s}(s)I_{\{s\}}(r),
\end{equation*}
where $t_{i}^{n}=t+\frac{i(T-t)}{n}, i=0, 1, \ldots,n$. Denote
\begin{align*}
&\Phi^{n}(\overline{\gamma})\deq \Phi(\gamma^{n}(\overline{\gamma})), \ \
  b^{n}(\overline{\gamma}_{s})\deq b(\gamma^{n}(\overline{\gamma}_{s})), \ \
  \sigma^{n}(\overline{\gamma}_{s})\deq \sigma(\gamma^{n}(\overline{\gamma}_{s})), \\
&f^{n}(\overline{\gamma}_{s},y,z)\deq f(\gamma^{n}(\overline{\gamma}_{s}),y,z), \ \
  g^{n}(\overline{\gamma}_{s},y,z)\deq g(\gamma^{n}(\overline{\gamma}_{s}),y,z).
\end{align*}
Then for each $n$, there exist some functions $\phi_{n}$ defined on $\Lambda_{t}\times \mathbb{R}^{n\times d}$,
$\varphi_{n}$ and $\widetilde{\varphi}_{n}$ defined on $[t,T]\times \Lambda_{t}\times \mathbb{R}^{n\times d}$,
and $\psi_{n}$ and $\widetilde{\psi}_{n}$ defined on
$[t,T]\times \Lambda_{t}\times \mathbb{R}^{n\times d}\times\mathbb{R}^{k}\times\mathbb{R}^{k\times d}$ such that
\begin{align*}
  \Phi^{n}(\overline{\gamma})&=\phi_{n}(\overline{\gamma}_{t},\overline{\gamma}(t_{1}^{n})-\overline{\gamma}(t),
                                       ...,\overline{\gamma}(t_{n}^{n})-\overline{\gamma}(t_{n-1}^{n})),\\
  b^{n}(\overline{\gamma}_{s})&=\varphi_{n}(s,\overline{\gamma}_{t},\overline{\gamma}(t_{1}^{n}\wedge s)-\overline{\gamma}(t),
                                       ...,\overline{\gamma}(t_{n}^{n}\wedge s)-\overline{\gamma}(t_{n-1}^{n}\wedge s)),\\
  \sigma^{n}(\overline{\gamma}_{s})&=\widetilde{\varphi}_{n}(s,\overline{\gamma}_{t},\overline{\gamma}(t_{1}^{n}\wedge s)-\overline{\gamma}(t),
                                       ...,\overline{\gamma}(t_{n}^{n}\wedge s)-\overline{\gamma}(t_{n-1}^{n}\wedge s)),\\
  f^{n}(\overline{\gamma}_{s},y,z)&=\psi_{n}(s,\overline{\gamma}_{t},\overline{\gamma}(t_{1}^{n}\wedge s)-\overline{\gamma}(t),
                                       ...,\overline{\gamma}(t_{n}^{n}\wedge s)-\overline{\gamma}(t_{n-1}^{n}\wedge s),y,z),\\
  g^{n}(\overline{\gamma}_{s},y,z)&=\widetilde{\psi}_{n}(s,\overline{\gamma}_{t},\overline{\gamma}(t_{1}^{n}\wedge s)-\overline{\gamma}(t),
                                       ...,\overline{\gamma}(t_{n}^{n}\wedge s)-\overline{\gamma}(t_{n-1}^{n}\wedge s),y,z).
\end{align*}
Indeed, if we set
\begin{align*}
\overline{\phi}_{n}(\overline{\gamma}_{t},x_{1},...,x_{n})
&\deq \Phi\bigg(\bigg(\overline{\gamma}_{t}(s)I_{[0,t)}(s) + \sum_{i=0}^{n-1}x_{i}I_{[t_{i}^{n},t_{i+1}^{n})}(s) + x_{n}I_{\{T\}}(s) \bigg)_{0\leq s\leq T}\bigg),\\
\phi_{n}(\overline{\gamma}_{t},x_{1},...,x_{n})
&\deq \overline{\phi}_{n}\bigg(\overline{\gamma}_{t},\overline{\gamma}(t)+x_{1},\overline{\gamma}(t)+x_{1}+x_{2},...,\overline{\gamma}(t)+\sum_{i=1}^{n}x_{i}\bigg),
\end{align*}
then from assumptions (H1)-(H4), we obtain that for each fixed $\overline{\gamma}_{t}$ and $i\in \{1,...,n\}$,
$$\phi_{n}^{i}(\overline{\gamma}_{t},x_{1},...,x_{n})\in \dbC_{l,\text{Lip}}^{0,2}(\mathbb{R}^{(n-1)d}\times \mathbb{R}^{d};\mathbb{R}^{k}),$$
where $\phi_{n}^{i}(\overline{\gamma}_{t},x_{1},...,x_{n})=\phi_{n}(\overline{\gamma}_{t},x_{1},...,x_{n},...,x_{i}), \ \forall x_{i}\in\mathbb{R}^{d}$.
In particular, for each $\overline{\gamma}\in\Lambda$,
\begin{equation*}
\partial_{x_{i}}\phi_{n}(\overline{\gamma}_{t},\overline{\gamma}(t_{1}^{n})-\overline{\gamma}(t),...,\overline{\gamma}(t_{n}^{n})-\overline{\gamma}(t_{n-1}^{n}))
=\Phi'_{\gamma_{t_{n-1}^{n}}}(\gamma^{n}(\overline{\gamma})).
\end{equation*}
Similar discussions can be used to $\varphi_{n},\widetilde{\varphi}_{n},\psi_{n}$ and $\widetilde{\psi}_{n}$.

\ms

Now for any $\overline{\gamma}_{\overline{t}}\in \Lambda_{\overline{t}}$, $\overline{t}\geq t$, consider the following SDE and BDSDE,
\begin{equation*}
  \begin{split}
    X^{n,\overline{\gamma}_{\overline{t}}}(s)
   =&\ \overline{\gamma}_{\overline{t}}(\overline{t}) + \int_{\overline{t}}^s b^{n}(X_{r}^{n,\overline{\gamma}_{\overline{t}}}) dr
    +\int_{\overline{t}}^s \sigma^{n}(X_{r}^{n,\overline{\gamma}_{\overline{t}}}) dW(r),\\
     Y^{n,\overline{\gamma}_{\overline{t}}}(s)
   =&\ \Phi^{n}(X_{T}^{n,\overline{\gamma}_{\overline{t}}})
    +\int_s^T f^{n}(X^{n,\overline{\gamma}_{\overline{t}}}_{r},Y^{n,\overline{\gamma}_{\overline{t}}}(r),Z^{n,\overline{\gamma}_{\overline{t}}}(r)) dr\\
    &+\int_s^T g^{n}(X^{n,\overline{\gamma}_{\overline{t}}}_{r},Y^{n,\overline{\gamma}_{\overline{t}}}(r),Z^{n,\overline{\gamma}_{\overline{t}}}(r)) d\oaB(r)
    -\int_s^T Z^{n,\overline{\gamma}_{\overline{t}}}(r) dW(r), \q~ \overline{t}\leq s.
  \end{split}
\end{equation*}
Denote
\begin{equation}\label{3.3-3}
  u^{n}(\overline{\gamma}_{\overline{t}})\deq Y^{n,\overline{\gamma}_{\overline{t}}}(\overline{t}), \q~ \overline{\gamma}_{\overline{t}}\in \Lambda.
\end{equation}
Iterating the same argument as Lemma \ref{14}, we get for each $s \in [t,T]$,
\begin{equation*}
  Z^{n,\overline{\gamma}_{\overline{t}}}(s) =
  \sigma^{n}(X^{n,\overline{\gamma}_{\overline{t}}}_{s})D_{x}u^{n}(X^{n,\overline{\gamma}_{\overline{t}}}_{s}), \q~ a.s.
\end{equation*}
Then, by the same calculus as in Theorem \ref{11}, and by the definition of $h^{n},b^{n},\sigma^{n},f^{n},g^{n}$, one can obtain that
\begin{align}\label{18}
 \lim_{n\rightarrow \infty} u^{n}(\overline{\gamma}_{\overline{t}})=u(\overline{\gamma}_{\overline{t}}), \q~
 \lim_{n\rightarrow \infty} D_{x}u^{n}(\overline{\gamma}_{\overline{t}})=D_{x}u(\overline{\gamma}_{\overline{t}}), \q~
 \lim_{n\rightarrow \infty} D_{xx}u^{n}(\overline{\gamma}_{\overline{t}})=D_{xx}u(\overline{\gamma}_{\overline{t}}), \q~ a.s.
\end{align}
Also,
\begin{align*}
 \lim_{n\rightarrow \infty} E\bigg[\sup_{s\in[t,T]} \big|D_{x}u^{n}(X^{n,\overline{\gamma}_{\overline{t}}}_{s})
    - D_{x}u(X^{\overline{\gamma}_{\overline{t}}}_{s})\big|^{2}\bigg]=0, \q~
 \lim_{n\rightarrow \infty} E\bigg[ \int_t^T \big|Z^{n,\overline{\gamma}_{\overline{t}}}(s)-Z^{\overline{\gamma}_{\overline{t}}}(s)\big|^{2}ds \bigg]=0.
\end{align*}
Therefore
\begin{equation*}
Z^{\gamma_{t}}(s)=\sigma(X^{\gamma_{t}}_{s}) D_{x}u(X^{\gamma_{t}}_{s}), \q~ \text{for every} \ \ s\in [t,T], \ a.s.
\end{equation*}
This completes the proof.
\end{proof}

Based on Theorem \ref{17}, the following result is directly.
\begin{corollary}\label{3.20.1}
From the above continuous version of $Z^{\gamma_t}$, for each $p\leq 2$, there is a constant $C_{p}$ and $q$ depending only on $C,T,m, \a$ and $p$ such that
$$|Z^{\gamma_t}(s)|\leq  C_p\big(1+\|X^{\gamma_t}_s\|^q\big),\q \forall s\in[t,T],\q a.s., $$
and
$$E\Big[\sup_{s\in[t,T]}|Z^{\gamma_t}(s)|^p\Big]\leq C_p(1+\|\gamma_t\|^q).$$
\end{corollary}

\section{Path-dependent SPDE}

It is well known that BDSDE with time-state dependent coefficients provide probabilistic representation for solution to semi-linear parabolic SPDE.
In this section, we shall establish the nonlinear stochastic Feynman-Kac formula for BDSDE with path-dependent coefficients,
which generates a new class of SPDE called path-dependent SPDE.
Based on this view, a (path-dependent) SPDE can be regarded as a BDSDE,
which provides an alternative approach for the study of SPDE theory and BDSDE theory.

\ms

We now relate our non-markovian BDSDE to the following system of path-dependent semilinear backward stochastic partial differential equation:
\begin{align}\label{SPDE}
 u(\gamma_{t})=&\   \Phi(\gamma_{T}) + \int_t^T \big[ \mathcal{L}u(\gamma_{s})
                  + f(\gamma_{s},u(\gamma_{s}),(\sigma D_{x}u)(\gamma_{s})) \big] ds \nonumber\\
               &\ + \int_t^T g(\gamma_{s},u(\gamma_{s}),(\sigma D_{x}u)(\gamma_{s})) d\oaB(s), \q~ 0\leq t\leq T,
\end{align}
where $u:\Lambda\rightarrow \mathbb{R}^{k}$, with
\begin{equation*}
  \mathcal{L}u = \frac{1}{2}tr\big[\big(\sigma\sigma^\top \big)D_{xx}u\big] + \langle b,D_{x}u \rangle.
\end{equation*}

\begin{theorem}\label{2.2.1} \sl
Let $u(\gamma_{t})\in \mathbb{C}^{1,2}_{l,\text{Lip}}(\Lambda)$ be a random field such that
$u(\gamma_{t})$ is $\mathcal{F}_{t,T}^{B}$-measurable for each $\gamma_{t}\in \Lambda$, and $u$ satisfies Eq. (\ref{SPDE}).
Then $u(\gamma_{t})=Y^{\gamma_{t}}(t)$, where $\big(Y^{\gamma_{t}}(s),Z^{\gamma_{t}}(s)\big)$ with $0\leq t\leq s\leq T$ is the unique solution of BDSDE (\ref{BDSDE}).
\end{theorem}
\begin{proof}
It suffices to show that $\{\big(u(X^{\gamma_{t}}_{s}), (\sigma D_{x}u)(X^{\gamma_{t}}_{s})\big);0\leq t\leq s\leq T \}$ solves BDSDE (\ref{BDSDE}).
Let $\pi:t=t_{0}<t_{1}<t_{2}<\cdot \cdot \cdot <t_{n}=T$ be a partition of the interval $[t,T]$.
Then we have that
\begin{align}\label{3.20.2}
  \sum_{i=0}^{n-1} \big[ u(X^{\gamma_{t}}_{t_{i}}) - u(X^{\gamma_{t}}_{t_{i+1}}) \big]
=& \sum_{i=0}^{n-1} \big[ u(X^{\gamma_{t}}_{t_{i}}) - u(X^{\gamma_{t}}_{t_{i},t_{i+1}}) \big]
  +\sum_{i=0}^{n-1} \big[ u(X^{\gamma_{t}}_{t_{i},t_{i+1}}) -  u(X^{\gamma_{t}}_{t_{i+1}}) \big].
 \end{align}
For the first term of the right side of \rf{3.20.2}, from  the equation satisfied by $u$, one has
\begin{align}
\sum_{i=0}^{n-1} \big[ u(X^{\gamma_{t}}_{t_{i}}) - u(X^{\gamma_{t}}_{t_{i},t_{i+1}}) \big]
= &\sum_{i=0}^{n-1} \bigg[ \int_{t_{i}}^{t_{i+1}}
     \big[ \mathcal{L}u(X^{\gamma_{t}}_{s}) + f(X^{\gamma_{t}}_{s},u(X^{\gamma_{t}}_{s}),(\sigma D_{x}u)(X^{\gamma_{t}}_{s})) \big]ds \nonumber \\
   & \ \ \ \ \ \ \ \ \ \     + \int_{t_{i}}^{t_{i+1}} g(X^{\gamma_{t}}_{s},u(X^{\gamma_{t}}_{s}),(\sigma D_{x}u)(X^{\gamma_{t}}_{s}))d\oaB(s) \bigg]. \label{3.20.3}
\end{align}
For the second term of the right side of \rf{3.20.2}, from the functional It\^{o}'s formula (see Theorem \ref{1}), one has
\begin{align}\label{3.20.4}
\sum_{i=0}^{n-1} \big[ u(X^{\gamma_{t}}_{t_{i},t_{i+1}}) -  u(X^{\gamma_{t}}_{t_{i+1}}) \big]
= -\sum_{i=0}^{n-1}\bigg[\int_{t_{i}}^{t_{i+1}}\mathcal{L}u(X^{\gamma_{t}}_{s}) ds+\int_{t_{i}}^{t_{i+1}}(\sigma D_{x}u)(X^{\gamma_{t}}_{s})dW(s)\bigg].
 \end{align}
Combining \rf{3.20.2}, \rf{3.20.3} and \rf{3.20.4}, we obtain that
\begin{align*}
  \sum_{i=0}^{n-1} \big[ u(X^{\gamma_{t}}_{t_{i}}) - u(X^{\gamma_{t}}_{t_{i+1}}) \big]
=& \sum_{i=0}^{n-1} \big[ u(X^{\gamma_{t}}_{t_{i}}) - u(X^{\gamma_{t}}_{t_{i},t_{i+1}}) \big]
  +\sum_{i=0}^{n-1} \big[ u(X^{\gamma_{t}}_{t_{i},t_{i+1}}) -  u(X^{\gamma_{t}}_{t_{i+1}}) \big]\\
= &\sum_{i=0}^{n-1} \bigg[ \int_{t_{i}}^{t_{i+1}}
     \big[ \mathcal{L}u(X^{\gamma_{t}}_{s}) + f(X^{\gamma_{t}}_{s},u(X^{\gamma_{t}}_{s}),(\sigma D_{x}u)(X^{\gamma_{t}}_{s})) \big]ds \\
   & \ \ \ \ \ \ \ \ \ \     + \int_{t_{i}}^{t_{i+1}} g(X^{\gamma_{t}}_{s},u(X^{\gamma_{t}}_{s}),(\sigma D_{x}u)(X^{\gamma_{t}}_{s}))d\oaB(s) \bigg]\\
&-\sum_{i=0}^{n-1}\bigg[\int_{t_{i}}^{t_{i+1}}\mathcal{L}u(X^{\gamma_{t}}_{s}) ds+\int_{t_{i}}^{t_{i+1}}(\sigma D_{x}u)(X^{\gamma_{t}}_{s})dW(s)\bigg].
 \end{align*}
Finally, by the existence and uniqueness theorem of BDSDE (\ref{BDSDE}), it finally suffices to let the mesh $\pi$ go to zero in order to conclude the result.
\end{proof}

By Theorem \ref{2.2.1} and Proposition \ref{1.8.9}, we have the following comparison theorem of path-dependent SPDE.

\begin{corollary}\sl
Assume $k=1$. Let $f=f_i$ and $\Phi=\Phi_i$ $(i=1,2)$ satisfy the conditions:
\begin{itemize}
  \item [{\rm (i)}] $f_1(\gamma, y,z)\geq f_2(\gamma, y,z)$ for all $(\gamma,y,z)\in\Lambda\times\dbR\times\dbR^d$;
  \item [{\rm (ii)}] $\Phi_1(\gamma_T)\geq \Phi_2(\gamma_T)$ for every $\gamma_T\in\Lambda_T$.
\end{itemize}
If $u_i(\gamma_{t})\in \mathbb{C}^{1,2}_{l,\text{Lip}}(\Lambda)$ is the solution of SPDE \rf{SPDE} associated with $(\Phi_1,f_1,g)$ and $(\Phi_2,f_2,g)$, $i=1,2$, respectively,
then $u_1(\gamma_t)\geq u_2(\gamma_t)$ for each $\gamma_t\in\Lambda$.
\end{corollary}

Now we prove the converse to the above result.

\begin{theorem}\label{2.2.1-2} \sl
 $\{u(\gamma_{t})\deq Y^{\gamma_{t}}(t), \gamma_{t}\in \Lambda\}$ is the unique $\mathbb{C}^{1,2}_{l,\text{Lip}}(\Lambda)$-solution of the path-dependent SPDE (\ref{SPDE}).
\end{theorem}

\begin{proof}
From Corollary \ref{13}, $u\in\mathbb{C}^{0,2}_{l,\text{Lip}}(\Lambda)$, a.s.
For each $\delta>0$ satisfying $t+\delta\leq T$, by Lemma \ref{16}, we can get
\begin{align*}
u(X_{t+\delta}^{\gamma_{t}})=Y^{\gamma_{t}}(t+\delta),\q~ \hb{a.s.}
\end{align*}
Hence, by the proof of Theorem \ref{17}, we obtain a.s.
\begin{align*}
     u(\gamma_{t,t+\delta})-u(\gamma_{t})
  =& u(\gamma_{t,t+\delta})-u(X_{t+\delta}^{\gamma_{t}})+u(X_{t+\delta}^{\gamma_{t}})-u(\gamma_{t})\\
  =& \lim_{n\rightarrow \infty} \big[ u^{n}(\gamma_{t,t+\delta})-u^{n}(X_{t+\delta}^{\gamma_{t}}) \big]
   - \int_t^{t+\delta} f(X_{s}^{\gamma_{t}},Y^{\gamma_{t}}(s),Z^{\gamma_{t}}(s)) ds\\
   &- \int_t^{t+\delta} g(X^{\gamma_{t}}_{s},Y^{\gamma_{t}}(s),Z^{\gamma_{t}}(s)) dB_{s} + \int_t^{t+\delta} Z^{\gamma_{t}}(s) dW_{s},
\end{align*}
where $u^{n}(\cd)$ is defined in \rf{3.3-3}.
Now recalling Theorem 3.2 of \cite{Peng94} and Theorem 3.2 of \cite{Peng92}, we deduce that
\begin{align*}
   u^{n}(\gamma_{t,t+\delta})-u^{n}(X_{t+\delta}^{\gamma_{t}})
=& \int_t^{t+\delta} D_{s}u^{n}(\gamma_{t,s}) ds - \int_t^{t+\delta} D_{s}u^{n}(X_{s}^{\gamma_{t}}) ds \\
 &-\int_t^{t+\delta} \mathcal{L}u^{n}(X_{s}^{\gamma_{t}}) ds-\int_t^{t+\delta} (\sigma^{n}D_{x}u^{n})(X_{s}^{\gamma_{t}})dW(s).
\end{align*}
Thus from (\ref{18}) and the dominated convergence theorem, we get
\begin{align}
     u(\gamma_{t,t+\delta})-u(\gamma_{t})
  =& - \int_t^{t+\delta} \mathcal{L}u(X_{s}^{\gamma_{t}}) ds-\int_t^{t+\delta}
  (\sigma D_{x}u)(X_{s}^{\gamma_{t}})dW(s) \nonumber \\
   & - \int_t^{t+\delta} f(X_{s}^{\gamma_{t}},Y^{\gamma_{t}}(s),Z^{\gamma_{t}}(s)) ds
     - \int_t^{t+\delta} g(X^{\gamma_{t}}_{s},Y^{\gamma_{t}}(s),Z^{\gamma_{t}}(s)) dB_{s} \nonumber  \\
   & + \int_t^{t+\delta} Z^{\gamma_{t}}(s) dW_{s}+\lim_{n\rightarrow \infty}C^{n}, \label{28}
\end{align}
where
\begin{align*}
C^{n} = \int_t^{t+\delta} D_{s}u^{n}(\gamma_{t,s}) ds - \int_t^{t+\delta} D_{s}u^{n}(X_{s}^{\gamma_{t}}) ds.
\end{align*}
Now using Remark \ref{29} and note that $u^{n}\in\mathbb{C}^{0,2}_{l,\text{Lip}}(\Lambda)$, a.s.,
there exists a constant $c$ depending only on $C,T,m,\a$ and $\gamma_t$ such that
\begin{align*}
| D_{s}u^{n}(\gamma_{t,s})-D_{s}u^{n}(X_{s}^{\gamma_{t}})|
\leq c\big(1+\sup_{u\in[t,s]}|X^{\gamma_{t}}(u)-X^{\gamma_{t}}(t)|^{m}\big) \|\gamma_{t,s}-X_{s}^{\gamma_{t}} \|.
\end{align*}
Therefore
\begin{align*}
|C^{n}|\leq c\delta \sup_{s\in[t,t+\delta]}\big[\big(1+|X^{\gamma_{t}}(u)-X^{\gamma_{t}}(t)|^{m}\big)|X^{\gamma_{t}}(s)-X^{\gamma_{t}}(t)|\big].
\end{align*}
Finally, taking expectation on both sides of (\ref{28}) yields that
\begin{align*}
\lim_{\delta\rightarrow 0} \frac{u(\gamma_{t,t+\delta})-u(\gamma_{t})}{\delta}
=-\mathcal{L}u(\gamma_{t})-f(\gamma_{t},u(\gamma_{t}),(\sigma D_{x}u) (\gamma_{t})).
\end{align*}
Therefore $u(\gamma_t)$ belongs to $\in\mathbb{C}^{1,2}_{l,\text{Lip}}(\Lambda)$, a.s., and satisfies path-dependent SPDE (\ref{SPDE}). This completes the proof.
\end{proof}

From the above theorem, the following result is directly.

\begin{corollary}\sl
The process $\big(u(X_{s}^{\gamma_{t}}),(\sigma D_{x}u)(X_{s}^{\gamma_{t}})\big)$ is the unique solution of BDSDE \rf{BDSDE}.
\end{corollary}

\begin{remark}\rm
In the case of that when $\Phi(\gamma_T)=\phi(\gamma(T))$ with $\phi\in\dbC_{l,\text{Lip}}^{2}(\mathbb{R}^{d};\mathbb{R}^{k})$ and $f$, $g$ satisfy Assumption (H5), the above theorems degenerated to the famous nonlinear stochastic Feynman-Kac formulas, which were introduced by Pardoux and Peng \cite{Peng94}.
\end{remark}

\section*{Acknowledgements}

The authors would like to thank Prof. Falei Wang, Prof. Christian Keller, the associate editor and the anonymous referees for their insightful comments that improved the quality of this paper.


\begin{thebibliography}{99}

\bibitem{Boufoussi}
   B.~Boufoussi, J.~Casteren and N.~Mrhardy, \it
   Generalized backward doubly stochastic differential equations and SPDEs with nonlinear Neumann boundary conditions, \sl
   Bernoulli, \rm
   13 (2007), 423--446.

\bibitem{Bismut}
  J.~Bismut, \it
  Conjugate convex functions in optimal stochastic control, \sl
  J. Math. Anal. Appl., \rm
  44 (1973), 384--404.

\bibitem{Buckdahn01}
   R.~Buckdahn and J.~Ma, \it
   Stochastic viscosity solutions for nonlinear stochastic partial differential equations. Part I, \sl
   Stochastic Process. Appl., \rm
   93 (2001), 181--204.

\bibitem{Buckdahn012}
   R.~Buckdahn and J.~Ma, \it
   Stochastic viscosity solutions for nonlinear stochastic partial differential equations. Part II, \sl
   Stochastic Process. Appl., \rm
   93 (2001), 205--228.

\bibitem{Buckdahn15}
   R.~Buckdahn, C.~Keller, J.~Ma and J.~Zhang, \it
   Pathwise Viscosity Solutions of Stochastic PDEs and Forward Path-Dependent PDEs-A Rough Path View, \sl
   arXiv:1501.06978, 2015.\rm

\bibitem{Cont10}
   R.~Cont and D.~Fourni\'{e}, \it
   Change of variable formulas for non-anticipative functional on path space, \sl
   J. Funct. Anal., \rm
   259 (2010), 1043--1072.

\bibitem{Cont13}
   R.~Cont and D.~Fourni\'{e}, \it
   Functional It\^{o} calculus and stochastic integral representation of martingales, \sl
   Ann. Probab., \rm
   41(1) (2013), 109--133.

\bibitem{Dupire09}
   B.~Dupire, \sl
   Functional It\^{o} calculus, \rm
   Portfolio Research Paper, Bloomberg, 2009.

\bibitem{Friz}
  J.~Diehl and P.~Friz, \it
  Backward stochastic differential equations with rough drivers, \sl
  Ann. Probab., \rm
  40 (2012), 1715--1758.

\bibitem{Ekren14}
   I.~Ekren, C.~Keller, N.~Touzi and J.~Zhang, \it
   On viscosity solutions of path dependent PDEs, \sl
   Ann. Probab., \rm
   42 (2014), 204--236.

\bibitem{Ekren16-1}
   I.~Ekren, N.~Touzi and J.~Zhang, \it
   Viscosity solutions of fully nonlinear parabolic path dependent PDEs: part I, \sl
   Ann. Probab., \rm
   44 (2016), 1212--1253.

\bibitem{Ekren16-2}
   I.~Ekren, N.~Touzi and J.~Zhang, \it
   Viscosity solutions of fully nonlinear parabolic path dependent PDEs: Part II, \sl
   Ann. Probab., \rm
   44 (2016), 2507--2553.

\bibitem{Han-Peng-Wu-10}
Y.~Han, S. Peng and Z. Wu,
\it Maximum Principle for Backward Doubly Stochastic Control Systems with Applications,\sl
\sl SIAM J. Control Optim.,\rm
\rm 48 (2010), 4224--4241.

\bibitem{Hu-Li-Wen-21}
   Y.~Hu, X.~Li and J. Wen, \it
   Anticipated backward stochastic differential equations with quadratic growth,\sl
   J. Differential Equations, 270 (2021), 1298--1331. \rm

%\bibitem{Jazaerli-Saporito-17}
%   S.~Jazaerli and Y.~Saporito, \it
%   Functional Ito calculus, path-dependence and the computation of Greeks, \sl
%   Stochastic Process. Appl., \rm
%   127 (2017), 3997--4028.

\bibitem{Ji-Yang-13}
   S.~Ji and S.~Yang, \it
   Classical Solutions of Path-Dependent PDEs and Functional Forward-Backward Stochastic Systems, \sl
   Math. Probl. Eng., \rm
   2013 (2013), Article ID 423101.

\bibitem{Lipster78}
  R.~Lipster and A.~Shiryaev, \sl
  Statistics of Random Processes I, \rm
  Springer, 1978.

%\bibitem{Nualart}
%  D.~Nualart and E.~Pardoux, \it
%  Stochastic calculus with anticipating integrands, \sl
%  Probab. Theory Related Fields, \rm
%  78 (1990), 535--581.

\bibitem{Peng90}
  E.~Pardoux and S.~Peng, \it
  Adapted solution of a backward stochastic differential equation, \sl
  Systems Control Lett., \rm
  4 (1990), 55--61.

\bibitem{Peng92}
  E.~Pardoux and S.~Peng, \it Backward stochastic differential equations and quasilinear parabolic partial differential equations, \rm In: B.~L.~Rozuvskii and R.~B.~Sowers (eds.) \sl Stochastic partial differential equations and their applications, Lect. Notes Control Inf. Sci., Berlin: Springer
   \rm 176 (1992), 200--217.

\bibitem{Peng94}
  E.~Pardoux and S.~Peng, \it
  Backward doubly stochastic differential equations and systems of quasilinear SPDEs, \sl
  Probab. Theory Relat. Fields, \rm
  98 (1994), 209--227.

\bibitem{Peng07}
  S.~Peng, \it
  G-expectation, G-Brownian Motion and Related Stochastic Calculus of It\^{o} type, \sl
  In: Stochastic Analysis and Applications. Abel Symp., Berlin: Springer, \rm 2 (2007), 541--567.

\bibitem{Peng91}
  S.~Peng, \it
  Probabilistic Interpretations for Systems of Quasilinear Parabolic Partial Differential Equation, \sl
  Stochastics, \rm
  37 (1991), 61--74.

\bibitem{Song15}
  S.~Peng and Y.~Song, \it
  G-expectation weighted Sobolev spaces, backward SDE and path dependent PDE, \sl
  J. Math. Soc. Japan, \rm
  67(4) (2015), 1725--1757.

\bibitem{Peng16}
   S.~Peng and F.~Wang, \it
   BSDE, path-dependent PDE and nonlinear Feynman-kac formula, \sl
   Sci. China Math., \rm
   59 (2016), 19-36.

\bibitem{Ren-Touzi-Zhang-20}
  Z. Ren, N. Touzi and J. Zhang, \it
  Comparison of viscosity solutions of semilinear path-dependent PDEs, \sl
  SIAM J. Control Optim., \rm
  58 (2020), 277--302.

\bibitem{Shi-Gu-Liu-05}
  Y. Shi, Y. Gu and K. Liu,\it
  Comparison Theorems of Backward Doubly Stochastic Differential Equations and Applications, \sl
  Anal. Appl., 23:97-110, 2005.\rm

\bibitem{Shi-Wen-Xiong-20}
  Y.~Shi, J.~Wen and J.~Xiong, \it
  Backward doubly stochastic Volterra integral equations and their applications, \sl
  J. Differential Equations, \rm
  269 (2020), 6492--6528.

\bibitem{Wen-Li-Xiong 21}
J. Wen, X. Li, J. Xiong, \it
Weak closed-loop solvability of stochastic linear quadratic optimal control problems of Markovian regime switching system, \sl
 Appl. Math. Optim., \rm
  84 (2021), 535--565.


\bibitem{Wen}
  J.~Wen and Y.~Shi, \it
  Backward doubly stochastic differential equations with random coefficients and quasilinear stochastic PDEs, \sl
  J. Math. Anal. Appl., \rm
  476 (2019), 86--100.


\bibitem{Xiong-13}
  J.~Xiong, \it
  Super-Brownian motion as the unique strong solution to an SPDE, \sl
  Ann. Probab., \rm
  41 (2013), 1030--1054.

\bibitem{Zhang-17}
  J.~Zhang, \sl
  Backward Stochastic Differential Equations. From linear to fully nonlinear theory. \rm
  Probability Theory and Stochastic Modelling, 86. Springer, New York, 2017.

\bibitem{Zhang-Zhao07}
   Q. Zhang and H. Zhao, \it
   Pathwise stationary solutions of stochastic partial differential equations and backward doubly stochastic differential equations on infinite horizon, \sl
   J. Func. Anal., \rm
   252 (2007), 171--219.

\bibitem{Zhang-Zhao10}
   Q. Zhang and H. Zhao, \it
   Stationary solutions of SPDEs and infinite horizon BDSDEs under non-Lipschitz coefficients, \sl
   J. Differential Equations, \rm
   248 (2010), 953--991.

\bibitem{Zong-Yin-Wang-Li-Zhang-18}
 X. Zong, G. Yin, L. Wang, T. Li and J. Zhang, \it
 Stability of stochastic functional differential systems using degenerate Lyapunov functionals and applications, \sl
 Automatica J. IFAC, \rm
 91 (2018), 197--207.

\end{thebibliography}
\end{document}